%% file: 210609-areaDH.tex
\documentclass[12pt,a4paper]{amsart}
\usepackage[english]{babel}
 \usepackage[utf8x]{inputenc} 
\usepackage{amsmath}
\usepackage{cmll}
\usepackage{amssymb}
\usepackage{amsthm}
\usepackage{xcolor}
\usepackage{bbm}
\usepackage{mathrsfs}
\usepackage[all]{xy}
\usepackage{indentfirst}
\usepackage{fancyhdr}
\usepackage{newlfont}
\usepackage{setspace}
\usepackage{verbatim}
\usepackage{hyperref} 
\usepackage{bbm}

\usepackage{geometry}
\include{def-Area-Heisenberg}

\newtheorem{teo}{Theorem}[section]
\newtheorem{prop}[teo]{Proposition}
\newtheorem{cor}[teo]{Corollary}
\newtheorem{lem}[teo]{Lemma}

\theoremstyle{definition}

\newtheorem{deff}{Definition}[section]

\newtheorem{Remark}{Remark}[section]

\begin{document}

\title[Area formula in Heisenberg groups]
{Area formula for regular submanifolds of low codimension in Heisenberg groups}

\author{Francesca Corni}
\address{Dip.to di Matematica, Universit\`a di Bologna, Piazza di Porta San Donato, 5, 40126, Bologna, Italy}
\email{francesca.corni3@unibo.it}
\author{Valentino Magnani}
\address{Dip.to di Matematica, Universit\`a di Pisa, Largo Bruno Pontecorvo 5, 56127, Pisa, Italy}
\email{valentino.magnani@unipi.it}

\thanks{The second author was supported by the University of Pisa, Project PRA 2018 49.}


\date{\today}


\begin{abstract}
We establish an area formula for the spherical measure of intrinsically
regular submanifolds of low codimension in Heisenberg groups. 
The spherical measure is constructed by an arbitrary homogeneous distance.
Among the arguments of the proof, we point out the differentiability properties of intrinsic graphs and a chain rule for intrinsic differentiable functions.
\end{abstract}

\maketitle

\tableofcontents
\section{Introduction}

Research in the Analysis and Geometry of simply connected nilpotent Lie groups has spread into several directions, especially in the last decade. Carnot groups, or stratified groups equipped with a homogeneous left invariant distance, are an important class of these nilpotent groups, which are metrically different from Euclidean spaces or Riemannian manifolds, still maintaining a rich algebraic and metric structure. 

Our aim is to compute the area of submanifolds in the Heisenberg
group $\H^n$, that represents the simplest model of 
noncommutative stratified group. For different classes of $C^1$ smooth submanifolds area formulas are available, see \cite{Magnani2019} and the references therein. The question has new difficulties when we consider ``intrinsic regular submanifolds'' of $\H^n$, that need not be $C^1$ smooth nor Lipschitz with respect to the Euclidean distance, \cite{KirSer04}.

On the other hand, they are suitable level sets of continuously differentiable functions from $\H^n$ to $\R^k$.
The differentiability here is understood with respect to the group operation and dilations, i.e.\ the so-called Pansu differentiability.
Precisely, these level sets are defined when $1 \leq k \leq n$ and the differential of the defining function is surjective.
They are called $\H$-regular surfaces of low codimension in $\H^n$ (Definition \ref{Hregularsurface}). 
These special submanifolds in $\H^n$ and their characterizations have been studied under different perspectives. We mention for instance the papers \cite{ASCV06}, \cite{Arena}, \cite{BSC10DistSol}, \cite{BSC10IntReg}, \cite{Areaformula}, along with the lecture notes \cite{Notesserra}
and references therein. 
An implicit function theorem, proved in \cite{Areaformula}, states that every $\H$-regular surface can be locally seen as an intrinsic graph with respect to a special semidirect factorization (Definition \ref{intgraph}). 
Although the parametrizing mapping of the $\H$-regular surface
is not Lipschitz continuous in the Euclidean sense, 
in \cite{Arena} the authors proved that it is uniformly intrinsic differentiable (Definition~\ref{uniformintdiff}). 
Indeed, uniform intrinsic differentiability for maps acting between
suitable factorizing homogeneous subgroups has been 
largely studied, also in a broader framework and from the viewpoint of
nonlinear first order systems of PDEs, \cite{ADDDLD20pr}, \cite{ADDD20pr}, \cite{Intsurfaces}, \cite{DiDonatoArt}, \cite{Artem}.

We consider a {\em vertical subgroup} $\W$ and a 
{\em horizontal subgroup $\V$} (Definition~\ref{df:horver}). 
We assume that $\H^n=\W\rtimes\V$ (Definition~\ref{def:semidir})
and we fix a {\em parametrized $\H$-regular surface $\Sigma$ with respect to $(\W,\V)$}, parametrized by $\phi$ and with defining function $f$ (Definition~\ref{def:param}).
The following measure $\mu$ can be associated to $\Sigma$.
For every Borel set $B\subset\Sigma$, we have
\begin{equation}\label{eq:areamuIntro}
	\mu(B)= \Vert V \wedge N \Vert_g \int_{\Phi^{-1}(B)} \frac{J_Hf(\Phi(n))}{J_\V f(\Phi(n))}  \  d \mathcal{H}_E^{2n+1-k} (n),
\end{equation}
where the Jacobians $J_Hf$ and $J_\V f$ are 
defined in \eqref{def:HJac} and \eqref{def:VJac},
$\Phi(n)=n\phi(n)$ is the graph map associated
to $\phi$ and $\cH^{2n+1-k}_E$ is the Euclidean Hausdorff measure.
The factor $\Vert V \wedge N \Vert_g$ takes into account
the ``angle'' between the multivectors $N$ and $V$, that are
associated with the domain and the codomain of the implicit mapping, 
respectively.

The measure in \eqref{eq:areamuIntro} was introduced in
\cite{Areaformula}, where the authors proved that it is equal
to the centered Hausdorff measure restricted to $\Sigma$.
Precisely, \cite[Theorem~4.1]{Areaformula}
has been revised in \cite[Theorem 4.50]{Notesserra}, 
using a metric area formula for the centered Hausdorff measure, \cite{FSSC8}.
The question of finding an area formula for the spherical measure of a low codimensional $\H$-regular surface remained unanswered.
The present paper settles this question, proving an area formula
for the spherical measure of $\Sigma$ in terms of the 
measure $\mu$.

Let $d$ be a fixed homogeneous distance in $\H^n$ and let us consider
the spherical measure $\cS^{2n+2-k}$ with respect to $d$,
according to \eqref{eq:Salpha}. 
We can associate a geometric constant $\beta_d(\Pi)$ 
to a $p$-dimensional subspace $\Pi$ and a distance $d$,
that is called {\em spherical factor}.
Essentially, it represents the maximal $p$-dimensional area
of the intersections of $\Pi$ with metric unit balls whose centers
are suitably close to the origin (Definition \ref{def:sphfactor}).
In our area formula, the spherical factor is computed for the homogeneous 
tangent cones $\Tan(\Sigma,x)$ of $\Sigma$ at the points $x\in\Sigma$ (Definition~\ref{def:tan}). 
Theorem~\ref{Maintheorem} establishes the 
``upper blow-up'' of the measure $\mu$, proving that the spherical factor 
of $\Tan(\Sigma,x)$ equals the $(2n+2-k)$-spherical Federer density of $\mu$ 
at $x\in\Sigma$ (Definition \ref{def:sphefederer}), namely
\begin{equation}\label{eq:equalitydensityspherical}
	\theta^{2n+2-k}( \mu, x)= \ \beta_{d}(\Tan(\Sigma,x)).
\end{equation}
The previous equality represents the central technical tool of the paper. 
Indeed, if we combine Theorem~\ref{Maintheorem} and the metric area formula of Theorem~\ref{abstractdiff}, we immediately obtain our main result,
that is the following area formula for the spherical measure.
\begin{teo}[Area formula]\label{teo:areaformulagenerale}
If $\Sigma$ is a parametrized $\H$-regular surface with respect
to $(\W,\V)$, then for every Borel set $B \subset \Sigma $ we have
\begin{equation}\label{eq:areadgenerica}
\mu(B)= \int_B \beta_{d}(\Tan(\Sigma,x)) \ d \mathcal{S}^{2k+2-k}(x),
\end{equation}
where the measure $\mu$ is defined in \eqref{eq:areamuIntro}.
\end{teo}
The previous integral formula also shows that the measure $\mu$ does not depend on the factorization $\W\rtimes\V$ and on the defining function $f$
appearing in \eqref{eq:areamuIntro}. 
In fact, when the factors $\W$ and $\V$ of the semidirect product of $\H^n$ are orthogonal, then \cite[Theorem~6.1]{Intsurfaces} proves that
the integrand in \eqref{eq:areamuIntro} can be written in terms of intrinsic partial derivatives of the parametrization $\phi$ of $\Sigma$,
namely the defining function $f$ disappears.
Thus, the area formula takes the following form.
\begin{teo}	\label{areaintder}
Let $\Sigma$ be a parametrized $\H$-regular surface with respect to $(\W,\V)$. Let $\phi$ be the parametrization of $\Sigma$
with respect to $(\W,\V)$, according to Definition~\ref{def:param}.
If $\Phi(n)=n\phi(n)$ is the graph mapping and 
$\W$ is orthogonal to $\V$, 
then for every Borel set $B \subset \Sigma $ we have
\begin{equation}		\label{eq:formuladerint}
\int_{ \Phi^{-1}(B)} J^{\phi} \phi(w) \ d \mathcal{H}_E^{2n+1-k}(w)=\int_B \beta_{d}(\Tan(\Sigma,x)) \ d \mathcal{S}^{2k+2-k}(x),
\end{equation}
where $J^{\phi} \phi$ is the natural intrinsic Jacobian of $\phi$
	(Definition \ref{intjacobian}).
\end{teo}
At this point, it is worth to give some ideas about the proof
of our main technical tool, that is the ``upper blow-up'' of 
Theorem~\ref{Maintheorem}. This type of blow-up appeared in 
codimension one, to compute the spherical Federer density of the perimeter measure, \cite{Magnani2017}.
In our higher codimensional framework, the proof of the upper blow-up involves some new features. Three key aspects must be emphasized. First, rather unexpectedly, we realize that the intrinsic differentiability of the parametrizing map $\phi$ (Theorem~\ref{t:ArenaSerapioni}) is crucial to establish the limit of the set \eqref{eq:setLambda}.
Second, we prove an ``intrinsic chain rule" (Theorem \ref{t:chain}) that permits us to connect the kernel of $Df$ with the intrinsic differential of $\phi$, according to \eqref{eq:graphChain}. 
However, to make our chain rule work we have slightly modified the well known notion of intrinsic differentiability associated to a factorization,
introducing the {\em extrinsic differentiability} (Definition~\ref{d:newintdif}). 
We will deserve more attention on this notion of differentiability
and its associated chain rule for next investigations, since they may have an independent interest.
Third, we establish a delicate algebraic lemma for computing the Jacobian of projections between two vertical subgroups that are complementary 
to the same horizontal subgroup (Lemma~\ref{l:MVPsubgroups}). 

The area formulas \eqref{eq:areadgenerica} and \eqref{eq:formuladerint}
take a simple form when the homogeneous distance $d$ is invariant under
suitable classes of symmetries.
We refer to $p$-vertically symmetric distances (Definition~\ref{def:vertsymm}) and multiradial distances (Definition~\ref{d:multiradial}).
For instance, the Cygan--Kor\'anyi distance, \cite{Cygan81}, the distances constructed in \cite[Theorem~2]{HebSik90} and 
the distance $d_\infty$ of \cite[Section~2.1]{Areaformula} are examples of multiradial distances.
Furthermore, the sub-Riemannian distance in the first Heisenberg group is 2-vertically symmetric.
Combining Theorem~\ref{teo:areaformulagenerale}, Theorem~\ref{fattorevertical} and Proposition~\ref{multiimpliesvert},
a simpler version of \eqref{eq:areadgenerica} can be immediately established.
\begin{teo}
	\label{Areaformula}
	Let $d$ be either a $(2n+1-k)$-vertically symmetric distance or
	a multiradial distance of $\mathbb{H}^n$. 
	Let $\Sigma$ be a parametrized $\H$-regular surface with respect
	to $(\W,\V)$ and let $\mu$ be defined as in \eqref{eq:areamuIntro}.
	We have that
	\begin{equation}
		\mu= \omega_d(2n+1-k) \mathcal{S}^{2k+2-k} \res \Sigma,
	\end{equation}
	where $\omega_d(2n+1-k)$ is the constant spherical factor
	introduced in Definition~\ref{not:rotpn}.
	Therefore, setting $\mathcal{S}^{2n+2-k}_d= \omega_d(2n+1-k) \mathcal{S}^{2n+1-k}$, we have
	\begin{equation}
		\label{eq5}
		\mathcal{S}_d^{2n+2-k}\res\Sigma =\mu=\Vert V \wedge N \Vert_g \ \Phi_{\sharp} \left(\frac{J_Hf }{J_{\V}f  } \circ \Phi    \right)\mathcal{H}^{2n+1-k}_{E} \res \mathbb{W} .
	\end{equation}
\end{teo}
In the assumptions of the previous theorem, assuming in addition that $\W$ and $\V$ are orthogonal, formula (\ref{eq5}) can be rewritten as follows
\begin{equation}\label{eq:areaparamSymm}
	\mathcal{S}_d^{2n+2-k} \res \Sigma (B)
	=\int_{\Phi^{-1}(B)}   J^{\phi} \phi(w) \ d \mathcal{H}_E^{2n+1-k}(w),
\end{equation}
for any Borel set $B\subset \H^n$, 
where $J^{\phi} \phi$ is the intrinsic Jacobian of $\phi$ (Definition~\ref{intjacobian}).
The form of formula \eqref{eq:areaparamSymm} clearly 
reminds of the Euclidean area formula. 
Indeed the Euclidean Hausdorff measure
$\cH_E^{2n+1-k}$ can be replaced by the
Lebesgue measure $\cL^{2n+1-k}$.

Some additional applications of our results concern the relationship between the spherical measure and the centered Hausdorff measure. 
This study is treated in Section~\ref{sect:area}, where the main result
is the equality between spherical measure and centered Hausdorff measure, assuming that the metric unit ball of the homogeneous distance
is convex (Theorem~\ref{coincidence}).

Concerning the more recent literature, a general form of the area formula can be written for suitably ``$C^1$ smooth" intrinsic graphs in stratified groups, \cite[Theorem~1.1]{JNGV20prArea}, using the Hausdorff measure or
the spherical measure. 
The proof mainly relies on a suitable application of measure theoretic area formulas \cite{Magnani2015}, (see also the developments of \cite{LecMag2020}).
The upper blow-up of the Hausdorff (or spherical) measure of an intrinsic graph leads to a natural notion of ``area factor" (\cite[Lemma~3.2]{JNGV20prArea}),
that formally represents the Jacobian of the graph mapping and
extends the notions of Jacobian introduced in 
\cite{AmbKir2000Rect}, \cite{Kir94} and \cite{Mag}.
The approach of \cite{JNGV20prArea} can be suitably adapted to obtain
area formulas involving the centered Hausdorff meaures, hence
using suitable ``centered area factors".
Then an area formula for the centered Hausdorff measure
of graphs of intrinsic Lipschitz mappings can be obtained
\cite[Theorem~1.3]{AntMer21}, under the assumption on their
a.e.\ intrinsic differentiability.

On the other hand, whenever a metric-algebraic notion of differentiability is available for the parametrization, it is reasonable to connect the measure of its image with its ``suitable differential'' by an explicit formula for the Jacobian, getting a full area formula. Connecting the Jacobian with the differential, hence allowing for an effective computation of the Hausdorff 
(or spherical) measure of a set, has been completely achieved
for intrinsic regular hypersyrfaces in stratified groups.
This result stems from the contribution of many authors,
see \cite{ADDDLD20pr} for the last version of this one-codimensional area formula,
along with the full list of references. 
For one codimensional intrinsic Lipschitz graphs area formulas for the spherical
measure are obtained in \cite{DiD20} for stratified groups of step two,
see also the references therein.

Parametrized intrinsic $\H$-regular submanifolds in Heisenberg groups essentially
represent the first higher codimensional case where 
the differential of the parametrization is connected to
the measure of the submanifold. 
For the centered Hausdorff measure we refer to the works
\cite{Intsurfaces} and \cite{Notesserra}.
The case of the spherical measure is more delicate and relies on the techniques described above, that lead to the upper blow-up of Theorem~\ref{Maintheorem}.
Although the problem of computing the ``area factor"  does not seem an easy task, due to the low regularity of
intrinsic graphs, we believe however that our scheme for the area formula in Heisenberg groups has actually a wider scope of applications.
For this reason, we have left such developments for subsequent investigations.

Finally, we wish to mention that combining Theorem~\ref{teo:areaformulagenerale} and the
deep Rademacher's theorem for intrinsic Lipschitz mappings in Heisenberg groups, \cite[Theorem~1.1]{Vit20pr}, our area formula extends to intrinsic Lipschitz graphs,
\cite[Theorem~1.3]{Vit20pr}.
A similar extension is not automatic in general, 
since an interesting example of nowhere intrinsic differentiable Lipschitz 
graph can be constructed, \cite{JNGV21prNowhere}.

\section{Definitions and preliminary results}\label{sec2}

The next sections will introduce notions, notations and basic tools that will be used
throughout the paper.

\subsection{Coordinates in Heisenberg groups}\label{sect:symplectic}

The purpose of this section is to introduce $(2n+1)$-dimensional Heisenberg groups, along with the special coordinates that allow us to identify $\H^n$ with $\R^{2n+1}$. The Heisenberg group $\mathbb{H}^n$ can be represented as a direct sum of two linear subspaces
\[
\mathbb{H}^n = H_1 \oplus H_2
\]
with dim($H_1$)$=2n$ and dim($H_2$)$=1$, endowed with a symplectic form $\omega$ on $H_1$ and a fixed nonvanishing element $e_{2n+1}$ of $H_2$.
We denote by $\pi_{H_1}$ and $\pi_{H_2}$ the canonical projections on $H_1$ and $H_2$, which are associated to the direct sum.

We can give to $\mathbb{H}^n$ a structure of Lie algebra by setting
\begin{equation}
[ p, q]= \omega(\pi_{H_1}(p),\pi_{H_1}(q)) \ e_{2n+1}.
\end{equation}
Then the Baker-Campbell-Hausdorff formula ensures that 
\begin{equation}
p q = p + q + \frac{[p,q]}{2}
\end{equation}
defines a Lie group operation on $\H^n$.
For $t>0$, the linear mapping 
$\delta_t:\H^n\to\H^n$ such that 
\[
\delta_t(w)=t^k w\quad \text{ if $w \in H_k$},
\]
$k=1,2$, defines {\em intrinsic dilation}.

Given $p \in \H^n$, we denote by $l_p$ the translation by $p$. Any left invariant vector field on $\mathbb{H}^n$ is of the form $ X_v(p)= dl_p(0)(v)$ for any $p\in\H^n$ and some $v \in \mathbb{H}^n$, where we have identified $\H^n$ with $T_0\mathbb{H}^n$.
Through the Baker-Campbell-Hausdorff formula, one can check that the Lie algebra of left invariant vector fields Lie($\mathbb{H}^n$) is isomorphic to the given Lie algebra $(\mathbb{H}^n, [ \cdot, \cdot])$.

We fix a symplectic basis $ (e_1, \dots, e_{2n} )$ of $(H_1,\omega)$, namely
\[
\omega(e_i,e_{n+j})=\delta_{ij},\quad \omega(e_i,e_j)=\omega(e_{n+i},e_{n+j})=0
\]
for every $i,j=1,\ldots,n$, where $\delta_{ij}$ is the Kronecker delta.
Thus, we have obtained a {\em Heisenberg basis} 
\[
\mathcal{B}= ( e_1, \dots, e_{2n+1} ),
\]
that allows us to identify $\mathbb{H}^n$ with $\mathbb{R}^{2n+1}$. The associated linear isomorphism is defined as
\begin{equation}
\label{ident}
\pi_{\mathcal{B}}: \mathbb{H}^n \to \mathbb{R}^{2n+1}, \ \ 
\pi_{\mathcal{B}}(p)= (x_1, \dots, x_{2n+1} )
\end{equation}
for $p=\sum_{j=1}^{2n+1}x_je_j$. 
We can read the given Lie product on $\mathbb{R}^{2n+1}$ as follows
\begin{equation*}
\begin{split}
[ (x_1, \ldots, x_{2n+1}), (y_1, \dots, y_{2n+1})]&= \pi_{\mathcal{B}}([ \sum_{i=1}^{2n+1}x_i e_i, \sum_{i=1}^{2n+1} y_i e_i])\\
&=\pa{0, \ldots, 0,\sum_{i=1}^n (x_i y_{i+n}-x_{i+n} y_i)}
\end{split}
\end{equation*}
then the group product takes the following form on $\mathbb{R}^{2n+1}$
\[
(x_1, \dots, x_{2n+1})  (y_1, \dots, y_{2n+1})= \pa{x_1 + y_1, \ldots, x_{2n+1} + y_{2n+1} + \sum_{i=1}^n \frac{x_iy_{i+n}-x_{i+n} y_i}2}.
\]
Taking into account the previous formula, in our coordinates we obtain the following basis of left invariant vector fields
\begin{equation}
\label{eqbase}
\begin{split}
 X_j(p) &= \partial_{x_j}-\frac{1}{2} x_{j+n} \partial_{x_{2n+1}} \ \ \ \ \ \ j=1, \dots, n\\
 Y_j(p) &= \partial_{x_{n+j}} + \frac{1}{2}x_j \partial_{x_{2n+1}} \ \ \ \ \ \ j=1, \dots, n\\
 T(p)&= \partial_{x_{2n+1}}.
\end{split}
\end{equation}
They clearly constitute a basis $( X_1, \dots, X_{2n+1} )$ of Lie($\mathbb{H}^n$) such that $X_j(0)=e_j$ for every $j=1,\ldots,2n+1$.
Any linear combination of $X_1,\ldots,X_{2n}$ is called a {\em left invariant horizontal vector field of $\H^n$}.

\subsection{Metric structure}\label{sect:metric}

We fix a scalar product $\ban{\cdot,\cdot}$ that makes our Heisenberg basis $\cB= (e_1,\ldots,e_{2n+1})$ orthonormal. In the sequel, any Heisenberg basis will be understood to be orthonormal.
We denote by $|\cdot|$ both the Euclidean metric on $\mathbb{R}^{2n+1}$ and the norm induced by $ \langle \cdot, \cdot \rangle $ on $\mathbb{H}^n$. 
The symmetries of the Heisenberg group $\H^n$ are detected through the isometry 
\[J:H_1\to H_1,\]
that is defined by the Heisenberg basis
\[
J(e_i)=e_{n+i}\quad\text{and}\quad J(e_{n+i})=-e_i
\]
for all $i=1,\ldots,n$. It is then easy to check that
\[
\langle p,q \rangle=\omega(p,Jq) \quad \text{and} \quad
J^2=-I
\]
for all $p,q\in H_1$.

A \textit{homogeneous distance} $d$ on $\mathbb{H}^n$ is a function $d: \mathbb{H}^n \times \mathbb{H}^n \to [0,+\infty)$ such that 
\[
d(zx,zy)=d(x,y)\quad \text{and}\quad  d(\delta_t(x), \delta_t(y))= t d(x,y)
\]
for every $x,y,z \in \mathbb{H}^n$ and $t >0$. 
Any two homogeneous distances are bi-Lipschitz equivalent.
We also introduce the {\em homogeneous norm} $\Vert x \Vert = d(x,0)$, $x\in\H^n$, associated to 
a homogeneous distance $d$. Notice that this norm satisfies
\[
\Vert xy\Vert\le \Vert x\Vert +\Vert y\Vert \quad\text{and}\quad \Vert\delta_r x\Vert=r\Vert x\Vert
\]
for $x,y\in\H^n$ and $r>0$.

By identifying $T_0\H^n$ with $\H^n$ and by left translating the fixed scalar product $\langle \cdot, \cdot \rangle$ on $\mathbb{H}^n$ we obtain a left invariant Riemannian metric $g$ on $\H^n$. Its associated Riemannian
norm is denoted by $\Vert \cdot \Vert_g$.
We may restrict the identification of $T_0\H^n$ with $\H^n$ to the so called horizontal 
subspace, by identifying $H_1$ with 
\[
H_0 \mathbb{H}^n\subset T_0\H^n.
\]
Then the horizontal fiber at $p\in\H^n$ is $H_p\mathbb{H}^n =dl_p(0) (H_0\H^n)$.  The collection of all horizontal fibers constitutes the so-called {\em horizontal subbundle} $H \mathbb{H}^n$.
If we restrict the left invariant metric $g$ to the horizontal subbundle $H \mathbb{H}^n$, we obtain a scalar product on each horizontal fiber, that is the sub-Riemannian metric. This leads in a standard way to the so-called {\em Carnot-Carath\'eodory distance}, or {\em sub-Riemannian distance}, \cite{Gromov1996}, that is an example of homogeneous distance.

\subsection{Differentiability and factorizations}

We have different notions of differentiability in $\H^n$ and general Carnot groups, starting from the notion of Pansu differentiability, \cite{Pansu}.
We fix throughout the paper a homogeneous distance $d$.
Let $\Omega\subset\H^n$ be an open set, let $f : \Omega  \to \mathbb{R}^k$, $x \in \Omega$ and $v \in H_1$.
If there exists 
$$ 
\lim_{t \to 0} \frac{f(x (t v))-f(x)}{t} \in \mathbb{R}^k,
$$
then we say that it is the \textit{horizontal partial derivative} at $x$ along $X_v$, that is
the unique left invariant vector field such that $X_v(0)=v$. The above limit is denoted by
$X_vf(x)$. Notice that $X_v$ is precisely a left invariant horizontal vector field. We say that $f \in C^1_h(\Omega, \mathbb{R}^k)$ if for every $x \in \Omega$ and 
every horizontal vector field $X\in\Lie(\H^n)$ the 
horizontal derivative $Xf(x)$ exists and it is continuous with respect to $x\in\Omega$.

A linear mapping $L:\mathbb{H}^n \to \mathbb{R}^k$ that is homogeneous,
i.e. $tL(v)=L(\delta_t v)$ for all $t>0$ and $v\in\H^n$, is an {\em h-homomorphism},
that stands for ``homogeneous homomorphism''.
If there exists an h-homomorphism $L:\H^n\to\R^k$ that satisfies 
$$
|f(x w)-f(x)-Df(x)(w)|= o ( d(w,0)) \qquad \ \  \text{as} \  \ d(w,0)  \to 0,
$$
then it is unique and it is called the {\em h-differential}, or {\em Pansu differential}, of $f$ at $x$.
We denote it by $Df(x)$. Notice that $f \in C^1_h(\Omega, \mathbb{R}^k)$ if and only if it is everywhere Pansu differentiable and $x \to Df(x)$ is continuous as a map from $\Omega$ to the space of h-homomorphisms, 
see for instance \cite[Section~3]{Magnanicoarea}. 



\begin{deff} 
Let $\Omega\subset\H^n$ be an open set and let $f \in C^1_h(\Omega, \mathbb{R})$. We call \textit{horizontal gradient} of $f$ at $x \in \Omega$ the unique vector $\nabla_Hf(x)$ of $H_1$ such that
$Df(x)(z)=\ban{\nabla_Hf(x),z}$ for every $z\in \H^n$.
\end{deff}

When differentiability meets the factorizations of Heisenberg groups, the notion of intrinsic differentiability comes up naturally, see \cite{Notesserra} for more information. Now we introduce some algebraic properties of factorizations in $\H^n$ in order to define intrinsic differentiability and its basic properties. 

\begin{deff} \label{df:horver}If a Lie subgroup of $\mathbb{H}^n$ is closed under intrinsic dilations, 
we call it a \textit{homogeneous subgroup}. Homogeneous subgroups of $\H^n$ containing $H_2$ are called {\em vertical subgroups}. Homogeneous subgroups contained in $H_1$ are called
{\em horizontal subgroups}.
\end{deff}

It is easy to realize that any homogeneous subgroup of $\H^n$ is either horizontal or vertical.
We also notice that normal homogeneous subgroups of $\H^n$ coincide with vertical subgroups.

\begin{deff}\label{def:semidir} Let $\W$ and $\mathbb{V}$ be a vertical subgroup and a horizontal subgroup of $\H^n$, respectively.
We say that $\H^n$ is the {\em semidirect product} of $\W$ and $\V$ if 
$\mathbb{H}^n= \mathbb{W}  \mathbb{V}$ and $\mathbb{W} \cap \mathbb{V} = \{0\}$. 
In symbols, we write $\H^n=\W\rtimes\V$.
\end{deff}

\begin{deff}\label{d:projWVM}
Let $\M$, $\W$ and $\V$ be homogeneous subgroups of $\mathbb{H}^n$ such that 
\beq\label{eq:factMVs}
\H^n=\M\rtimes \V=\W\rtimes\V.
\eeq
The semidirect product $\W\rtimes\V$ automatically yields the unique projections 
\[
\pi_\W:\H^n\to\W\quad\text{ and }\quad \pi_\V:\H^n\to\V
\]
such that $x=\pi_\W(x)\pi_\V(x)$ for every $x\in\H^n$. If necessary, to emphasize the dependence on the
semidirect factorization we will also introduce the notation
$\pi^{\W,\V}_\W=\pi_\W$ and $\pi^{\W,\V}_\V=\pi_\V$. The same holds for $\M \rtimes \V$.
We define the following restrictions 
\[
\pi_{\W,\M}^{\W,\V}=\pi^{\W,\V}_\W|_{\M}:\M\to \W\quad\text{and}\quad 
\pi_{\M,\W}^{\M,\V}=\pi^{\M,\V}_\M|_{\W}:\W\to \M
\]
\end{deff}

\begin{Remark}
The uniqueness of the factorizations \eqref{eq:factMVs} implies that both restrictions
$\pi^{\W,\V}_{\W,\M}$ and $\pi^{\M,\V}_{\M, \W}$ are invertible and 
\beq\label{eq:invFactor}
\pi_{\W,\M}^{\W,\V}=(\pi_{\M,\W}^{\M,\V})^{-1}.
\eeq
\end{Remark}
If $\H^n=\mathbb{W}\rtimes \V$, then by the local compactness of $\H^n$, it is immediate to observe that
there exists a constant $c_0\in(0,1)$, possibly depending on $\W$ and $\V$, such that for all 
$w \in \mathbb{W}$ and $v \in \mathbb{V}$ the following holds
\begin{equation}\label{eq:estimnormsub}
c_0 \ ( \Vert w \Vert + \Vert v \Vert ) \leq \Vert wv \Vert \leq \Vert w \Vert + \Vert v \Vert .
\end{equation}

\begin{Remark}
Whenever two homogeneous subgroups $\W$ and $\V$ of $\H^n$ satisfy 
\[
\H^n=\W\V\quad\text{and} \quad \W\cap \V=\set{0},
\]
then one of them must be necessarily vertical and the other one must be horizontal. 
\end{Remark} 

Now we recall some results and definitions about intrinsic graphs of functions between two homogeneous subgroups. 
In the sequel $\W$ and $\V$ denotes a vertical subgroup and a horizontal
subgroup, respectively, such that $\H^n=\W\rtimes\V$.
For more information, see \cite{Notesserra}.

\begin{deff}
\label{intgraph}
For a nonempty set $U \subset\W$ and $\phi: U \to \mathbb{V}$
we define the \textit{intrinsic graph of $\phi$} as the set
$$ \mathrm{graph}(\phi)=\{w\phi(w):  w \in U \}.$$
We also introduce the {\em graph map $\Phi: U \to \Sigma$ of $\phi$}
by $\Phi(w)= w\phi(w)$ for all $w\in U$.
\end{deff}

\begin{Remark}
It is important to observe that the notion of intrinsic graph
is invariant with respect to both translations and dilations.
\end{Remark}

To study the action of translations on intrinsic graphs,
we need the following definition.

\begin{deff}
Let us consider $x \in \H^n$. We define $\sigma_x:\W \to \W$ as follows
\[
\sigma_x(w)=\pi_{\W}(l_x(w))= xw(\pi_\V(x))^{-1}
\]
for every $w\in \W$.
Given a set $U \subset \W$ and a function $\phi:U  \to \V$, the {\em translation of $\phi$ at $x$}, $\phi_{x}: \sigma_x(U) \to\V$ is defined as  
\begin{equation}\label{eq:translatedf}
\phi_x(w)=\pi_\V(x)\phi(x^{-1}w\pi_\V(x))=\pi_\V(x)\phi(\sigma_{x^{-1}}(w)).
\end{equation}
\end{deff}

\begin{Remark}
The map $\sigma_x$ is invertible on $\W$
\[
\sigma_{x^{-1}}(w)=x^{-1}w\pi_\V(x^{-1})^{-1}=
x^{-1}w\pi_\V(x)=\sigma_x^{-1}(w).
\]
Then for $w\in \sigma_x(U)$, we may also write
\begin{equation}
\phi_x(w)=\pi_\V(x)\phi(\sigma_{x}^{-1}(w)).
\end{equation}
\end{Remark}

Next we recall the content of \cite[Propositions 3.6]{Arena}.
\begin{prop}
Let $U \subset \W$ be an open set and $\phi: U  \to \V$ be a function. Then we have 
$$ 
l_x(\mathrm{graph}(\phi))= \{ w\phi_x(w) : w \in \sigma_x(U) \}.
$$
\end{prop}

\begin{deff}\label{d:intrinsicDiff}
Let $U \subset \W$ be an open set and let $\phi: U  \to \V$ be a function. Let us take $\bar{w} \in U$ and define $x=\bar{w} \phi(\bar{w})$. The function $\phi$ is intrinsic differentiable at $\bar{w}$ if there exists an h-homomorphism $L: \W \to \V$ such that 
\begin{equation}
\label{intdiff}
d(L(w), \phi_{x^{-1}}(w)) =o(\| w \|)
\end{equation} 
as $w\to 0$.
The function $L$ is called the {\em intrinsic differential} of $\phi$ at $\bar{w}$, it is uniquely defined
and we denote it by $d \phi_{\bar{w}}$.

\end{deff}

\begin{Remark}
By virtue of \cite[Proposition 3.23]{Arena}, in our setting any intrinsic linear function is actually an h-homomorphisms. We also observe that the assumption $\bar w\in U$ implies that 
$0\in\sigma_{x^{-1}}(U)$. In addition $\sigma_{x^{-1}}(U)$ is an open set, hence the limit
\eqref{intdiff} is entirely justified.
\end{Remark}

\begin{Remark} 
By \cite[Proposition~3.25]{Arena}, condition \eqref{intdiff} is equivalent to ask that for all $w \in U$
$$ \| d \phi_{\bar{w}}(\bar{w}^{-1} w)^{-1} \phi(\bar{w})^{-1} \phi(w) \|=o ( \| \phi(\bar{w})^{-1} \bar{w}^{-1} w \phi(\bar{w}) \|) $$ as
$\| \phi(\bar{w})^{-1} \bar{w}^{-1} w \phi(\bar{w}) \| \to 0$.
\end{Remark}

\begin{deff}
\label{uniformintdiff}
Let $U \subset \W$ be an open set and let $\phi: U  \to \V$ be a function. The map $\phi$ is uniformly intrinsic differentiable on $U$ if for any point $\bar{w} \in U$ there exists an h-homomorphism $d\phi_{\bar{w}}: \W \to \V$ such that
\begin{equation}
\lim_{\delta \to 0}\sup_{\|\bar{w}^{-1} w' \| < \delta} \sup_{0< \| w \|< \delta
 } \frac{ d(d \phi_{\bar{w}}(w),\phi_{\Phi(w')^{-1}}(w))}{\| w \|} =0
\end{equation}
where $\Phi$ is the graph map of $\phi$.
\end{deff}


%
The following definition is a slight modification of the notion of intrinsic differentiability.

\begin{deff}\label{d:newintdif}
Let $U\subset\W$ be an open set and let $F: U \to\R^k$ with $u\in U$.  
We choose $v\in \V$ and define $x=uv\in\H^n$ and the corresponding translated function
\[
F^{\V}_{x^{-1}}(w)=F(\sigma_x(w))-F(u)
\]
for $w\in \sigma_{x^{-1}}(U)$.
We say that $F$ is {\em extrinsically differentiable at $u$ with
respect to $(\V,x)$} if there exists an h-homomorphism $L:\W\to\R^k$ such that 
\beq\label{d:extdif}
\frac{|F_{x^{-1}}^{\V}(w)-L(w)|}{\|w\|}\to0\quad\text{as $w \to0$.}
\eeq
	The uniqueness of $L$ allows us to denote it by $d^{\V}_xF$.
\end{deff}
The terminology {\em extrinsic differentiabilty} arises from the fact that the subgroup $\V$ and the point $x$ cannot be detected from the information we have on $F$. In a sense, they are ``artificially added from outside''. 

\begin{Remark} 
If in the previous definition we embed $\R^k$ in $\H^n$, hence
replacing it by $\V$, and choose $v=F(u)$, then $x=uF(u)\in\H^n$ and
we have the equalities
\[
F^\V_{x^{-1}}(w)=F(\sigma_x(w))-v=F_{x^{-1}}(w).
\]
Thus, the numerator of \eqref{d:extdif} 
becomes equivalent to $d(F_{x^{-1}}(w),L(w))$ and the extrinsic differentiability of $F$ at $u$ with respect to 
$(\V,x)$ coincides with the intrinsic differentiability of $F$ at $u$. 
\end{Remark}

Extrinsic and intrinsic differentiability compensate each other in the following theorem.

\begin{teo}[Chain rule]\label{t:chain}
Let us consider two open sets $U \subset \W$, $\Omega \subset\H^n$ and two functions $f : \Omega \to \R^k$, $\phi: U \to \V$. Assume $\Phi(U) \subset \Omega$ where $\Phi$ is the graph function of $\phi$. Let us consider $x_{\W} \in U$ and set $x= \Phi(x_{\W})$. If $f$ and $\phi$ are h-differentiable at $x$ and intrinsic differentiable at $x_{\W}$, respectively,
then the composition $F=f\circ\Phi:U \to\R^k$, given by
$$
F(u)=f(u\phi(u))\quad \text{for all $u\in U$},
$$
is extrinsically differentiable at $x_{\mathbb{W}}$ with respect to
$(\V,x)$. For every $w \in \W$ the formula 
\beq\label{eq:diffchain}
d^{\V}_x F(w)= Df(x)(w d \phi_{x_{\W}}(w))
\eeq
holds. If in addition $f( w\phi(w))=c$ for every $w \in U$ and some $c\in\R$, then we obtain
\begin{equation}\label{eq:kernels}
\ker(Df(x))=\mathrm{graph}(d \phi_{x_{\W}}).
\end{equation}
\end{teo}

\begin{proof}
Let us first show that $F$ is extrinsically differentiable at $x_{\W}$ with respect to  $(\V,x)$. We define 
\[
L(w)=Df(x)(w d\phi_{x_{\W}}(w))=Df(x)(w) + Df(x)(d \phi_{x_{\W}}(w))
\]
for $w \in \W$, that is an h-homomorphism. For $w$ small enough, we have
\begin{equation*}
\begin{split}
\frac{|F^\V_{x^{-1}}(w)-L(w)|}{\| w\|} &= \frac{|f(x w x_{\V}^{-1} \phi(x w x_{\V}^{-1}))-f(x)-L(w)|}{\|w\|}\\
&= \frac{|f(x w \phi_{x^{-1}}(w))-f(x)-Df(x)(w d \phi_{x_{\W}}(w))|}{\| w \|}\\
& \leq \frac{|f(x w \phi_{x^{-1}}(w))-f(x)-Df(x)(w \phi_{x^{-1}}(w))|}{\|w\|}\\
&+ \frac{|Df(x)(w \phi_{x^{-1}}(w))-Df(x)(w d \phi_{x_{\W}}(w))|}{\| w \|}.
\end{split}
\end{equation*}
Let us consider the last two addends separately:
\begin{equation*} 
\begin{split}
&\frac{|f(x  w \phi_{x^{-1}}(w))-f(x)-Df(x)(w \phi_{x^{-1}}(w))|}{\|w\|}\\
&=\frac{|f(x w \phi_{x^{-1}}(w))-f(x)-Df(x)(w \phi_{x^{-1}}(w))|}{\|w \phi_{x^{-1}}(w)\|}\,\frac{\|w \phi_{x^{-1}}(w)\|}{\| w \|} \to 0
\end{split}
\end{equation*}
as $ \| w \| \to 0$, by the Pansu differentiability of $f$ at $x$ and by the 
validity of
$$ \frac{ \|w \phi_{x^{-1}}(w)\|}{\| w \|} \leq 1 + \frac { \| \phi_{x^{-1}}(w) \|}{\| w \|} = 1 + \left\|  \ d \phi_{x_{\W}}\left( \frac{w}{\| w \|} \right) \right\|+ \frac { \|  d \phi_{x_{\W}}(w)^{-1}  \phi_{x^{-1}}(w) \|}{\| w \|} \leq C_x$$
for all $w\neq0$ and sufficiently small. It is indeed a consequence of the intrinsic differentiability of $\phi$ at $x_{\W}$. For the second addend, the previous intrinsic differentiability yields 
$$\frac{|Df(x)(d \phi_{x_{\W}}(w)^{-1} \phi_{x^{-1}}(w))|}{\|w\|} = \left| Df(x) \left( \frac{  d \phi_{x_{\W}}(w)^{-1} \phi_{x^{-1}}(w))  }{\| w \|} \right) \right| \to 0 $$
as $ w\to0$. This complete the proof of the first claim 
and also establishes formula \eqref{eq:diffchain}.

Let us now assume the constancy of $w\to f(w \phi(w))$ on $U$.
Since we have proved that $F$ is extrinsically differentiable at $x_{\mathbb{W}}$ with respect to $(\V,x)$, being in this case $F^\V_{x^{-1}}$ identically
vanishing, we obtain
\[
d^{\V}_xF(w)=o(\|w\|)
\]
as $w\to0$. Therefore, for any $u \in \W$, we have
\begin{equation*}
\| Df(x)(\delta_tu d \phi_{x_{\W}}(\delta_tu)) \|= o ( t ) 
\end{equation*}
as $t \to 0$. Due to the h-linearity, it follows that
\begin{equation*}
 Df(x)(u d \phi_{x_{\W}}(u))=0.
\end{equation*}
We have proved the inclusion $\mathrm{graph}(d \phi_{x_{\W}}) \subset \ker(Df(x))$ of homogeneous subgroups with the same dimension, hence
formula \eqref{eq:kernels} is established.
\end{proof}

The notion of $\mathbb{H}$-regular surface in $\mathbb{H}^n$ was first given in \cite{Areaformula}.

\begin{deff}
\label{Hregularsurface}
Let $\Sigma \subset \mathbb{H}^n$ be a set and let $1 \leq k \leq n$. We say that $\Sigma$ is an $\mathbb{H}$-\textit{regular surface of low codimension},
or a {\em $k$-codimensional $\H$-regular surface}, if for every $x \in \Sigma$ there exist an open set $\Omega$ containing $x$ and a function $f =(f_1,\ldots, f_k) \in C^1_{h}(\Omega, \mathbb{R}^k)$ such that
\begin{itemize}
\item [(i)]$\Sigma \cap \Omega = \{ y \in \Omega : f(y)=0 \}$,
\item [(ii)]$\nabla_H f_1(y) \wedge \dots \wedge \nabla_Hf_k (y)\neq 0$  for all $y \in \Omega.$ 
\end{itemize}
\end{deff}
We can characterize the metric tangent cone of an $\mathbb{H}$-regular surface of codimension $k$.
\begin{deff}\label{def:tan}
For $A\subset\H^n$ and $x\in A$, the {\em homogeneous tangent cone} is the set
\[
\Tan(A,x)=\set{\nu\in\H^n:\,\nu=\lim_{h\to\infty} \delta_{r_h}(x^{-1}x_h),\, r_h>0,\,x_h\in A,\,x_h\to x}.
\]
\end{deff}
From \cite[Proposition~3.29]{Areaformula}, we have the following characterization.
\begin{prop}\label{prop:Tan}
If $\Sigma$ is an $\H$-regular surface of low codimension and $f\in C^1_h(\Omega,\R^k)$
is as in Definition~\ref{Hregularsurface}, then 
\[
\ker Df(x)=\Tan(\Sigma,x)
\]
for all $x\in \Sigma\cap\Omega$.
\end{prop}
Given an open subset $\Omega \subset \H^n$, a function $f \in C^1_h(\Omega, \R^k)$ and $x \in \Omega$,
we define the {\em horizontal Jacobian}
\begin{equation}\label{def:HJac}
J_Hf(x)=\|\nabla_H f_1(x) \wedge \dots \wedge \nabla_Hf_k (x)\|_g,
\end{equation}
where the norm is given through our fixed left invariant metric $g$.

If $f \in C^1_h(\Omega, \R)$ and $\V\subset H_1$ is a $k$-dimensional subspace, we set $\nabla_{\V}f(x)$ as the unique vector of $\V$ such that
$Df(x)(z)=\ban{\nabla_{\V}f(x),z}$ for every $z\in \V$. As a consequence, we can also define the {\em Jacobian with respect to $\V$}, namely
\begin{equation}\label{def:VJac}
J_{\V}f(x)=\|\nabla_{\V} f_1(x) \wedge \dots \wedge \nabla_{\V} f_k (x)\|_g.
\end{equation}
The next implicit function theorem is proved in \cite[Theorem 3.27]{Areaformula}.
Its general version in the framework of homogeneous groups is given in \cite[Theorem~1.3]{Magnani_2013}.

\begin{teo}[Implicit function theorem]\label{IFT}
Let $\Omega \subset \mathbb{H}^n$ be an open set, let $f \in C^1_h(\Omega, \mathbb{R}^k)$ be a function and consider a point $x_0 \in \Omega$ such that
$J_{\V}f(x_0) > 0$.
We define the level set
$$
\Sigma= \{ x \in \Omega :  f(x)=f(x_0) \}.
$$
Setting $\pi_\W(x_0)=\eta_0$ and $\pi_\V(x_0)=\up_0$,
there exist an open set $\Omega' \subset \Omega \subset \mathbb{H}^n$, with $x_0 \in \Omega'$, an open set $U \subset \mathbb{W}$ with $\eta_0 \in U$ and a unique continuous function
$ \phi: U \to \mathbb{V}$ such that $\phi(\eta_0)=\up_0$ and 
$$ \Sigma \cap \Omega'= \{w \phi(w) : w \in U \}.$$
\end{teo}

\begin{deff}[Parametrized $\H$-regular surface]\label{def:param}
Let $\Sigma$ be an $\H$-regular surface. We assume that 
there exist a semidirect factorization $\H^n=\W\rtimes\V$,
an open set $U\subset\W$ and a uniformly intrinsic differentiable $\phi:U\to \V$ such that $$\Sigma=\set{u\phi(u)\in\H^n: u\in U}.$$
We say that $\Sigma$ is a {\em parametrized $\H$-regular surface with
respect to $(\W,\V)$}, where $\phi$ is the {\em parametrization of $\Sigma$}.
If $\Omega\subset\H^n$ is open, $\Sigma\subset\Omega$, and we have $f\in C^1_h(\Omega,\R^k)$, $x_0\in\Sigma$, 
such that $f^{-1}(f(x_0))\cap U\V=\Sigma$ and $Df(x):\H^n\to\R^k$
is surjective for every $x\in\Sigma$, then we say that $f$ is a 
{\em defining function of $\Sigma$}.
\end{deff}

\begin{prop}\label{prop:parametrJV}
Let $\Omega\subset\H^n$ be open and let $f\in C^1_h(\Omega,\R^k)$ be such that $f^{-1}(f(x_0))=\Sigma$ for some $x_0\in\Omega$. 
If $J_\V f(x)>0$ for all $x\in\Sigma$, then $\Sigma$ is a parametrized $\H$-regular surface 
with respect to $(\W,\V)$ and $f$ is a defining function.
\end{prop}
\begin{proof}
We may apply the implicit function theorem of \cite[Proposition 3.13]{Areaformula}
at any point $x\in\Sigma$. Then locally $\Sigma$ is an intrinsic graph,
and by the uniqueness of the implicit mapping we can conclude that
$\Sigma$ actually is entirely parametrized by a unique graph mapping.
The uniform intrinsic differentiability of this parametrization 
follows from \cite[Theorem~4.2]{Arena}.
As a result, $\Sigma$ is a parametrized $\H$-regular surface
with respect to $(\W,\V)$ and $f$ is its defining function.
\end{proof}

A simple application of Theorem~\ref{t:chain}
is the following proposition.

\begin{prop}
Let $U\subset\W$ be open and assume that $\phi:U\to \V$ is
everywhere intrinsic differentiable. Let $\Sigma=\set{n\phi(n):n\in U}$ 
and let $\Omega\subset\H^n$ be open such that $\Sigma\subset\Omega$.
If $f:\Omega\to\R^k$ is everywhere h-differentiable with 
\[
\Sigma=f^{-1}(f(x_0))\cap(U\V)
\]
for some $x_0\in U\V$ and $J_Hf(x)>0$ for all $x\in\Sigma$, then $J_\V f(x)>0$
for all $x\in\Sigma$.
\end{prop}
\begin{proof}
We consider $x=w\phi(w)$, so by Theorem~\ref{t:chain} the function 
$F=f\circ \Phi$ is extrinsically differentiable at $w$ with respect to
$(\V,x)$ and 
\[
0=d^{\V}_x F(v)= Df(x)(v d \phi_{x_{\W}}(v))=D_\W f(x)(v)+D_\V f(x)(d \phi_{x_{\W}}(v))
\]
where $v\in\W$ and $D_Sf(x)=Df(x)|_S$ for any homogeneous subgroup $S$ of $\H^n$.
If by contradiction $D_\V f(x):\V\to\V$ would not be a isomorphism, then its image
$T$ would have linear dimension less than $k$. Then the previous equalities would imply
that the image of $D_\W f(x)$ would be contained in $T$, hence the same would hold
for the image of $Df(x)$. This conflicts with the fact that $Df(x)$ is surjective.
\end{proof}

The following corollary is a straightforward consequence of 
the previous proposition.

\begin{cor}
If $\H^n=\W\rtimes\V$ is a semidirect product and 
$f$ is a defining function of a parametrized $\H$-regular surface $\Sigma$
with respect to $(\W,\V)$, then $J_\V f(x)>0$ for every $x\in \Sigma$.
\end{cor}

We conclude this section by pointing out that the intrinsic graph in the 
implicit function theorem is suitably differentiable.

\begin{teo}[{\cite[Theorem 4.2]{Arena}}]
\label{t:ArenaSerapioni}
In the assumption of Theorem~\ref{IFT}, $\phi$ is uniformly intrinsic differentiable on $U$.
\end{teo}

\subsection{Intrinsic derivatives}
In this section we recall some results about uniform intrinsic differentiability in Heisenberg groups. Throughout this section, we assume that $\H^n$
is a semidirect product $\W\rtimes \V$ with $\W$ orthogonal to $\V$.
The following proposition ensures that we can always find
a Heisenberg basis which is adapted to this factorization.

\begin{prop}
\label{Heisbasis}
We assume that $\V$ is spanned by an orthonormal basis $(v_1, \dots, v_k)$. 
Then $k\le n$ and there exist an orthonormal basis 
$(v_{k+1}, \dots, v_n, w_1, \dots, w_n, e_{2n+1})$ of $\W$ 
such that $(v_1, \dots, v_n, w_1, \dots, w_n, e_{2n+1})$ 
is a Heisenberg basis of $\H^n$.
\end{prop}
\begin{proof}
Since $\V$ is commutative, an element $v=J(w)$ with $v,w\in\V$ satisfies
\[
|v|^2=\ban{v,J(w)}=-\omega(v,w)=0,
\]
therefore $\V \cap J(\V)= \set{0}$.
We set $w_i=J(v_i) \in \W$ for $i=1, \dots, k$ and define the
$2k$-dimensional subspace
\[
\S_{1}=\V\oplus J(\V)\subset H_1.
\]
We notice that
$
\dim(\S_{1}^{\bot}\cap H_{1})=2(n-k).
$
If $k<n$, we pick a vector $v_{k+1}\in\S_{1}^{\bot}\cap H_{1}$ 
of unit norm and define $w_{k+1}=Jv_{k+1}$.
It is easily observed that both $w_{k+1}$ and $v_{k+1}$ 
are orthogonal to $\S_1$, so that 
$(v_1,\ldots,v_{k+1},w_1,\ldots,w_{k+1},e_{2n+1})$ is a Heisenberg 
basis of 
\[
\S_2\oplus \spn\set{e_{2n+1}},
\]
where we have defined $\S_{2}=\V\oplus\spn\set{v_{k+1}}\oplus J(\V\oplus\spn\set{v_{k+1}}$). Indeed, the previous subspace has
the structure of a $(2k+3)$-dimensional Heisenberg group.
One can iterate this process until a Heisenberg basis of
$\H^n$ is found.
\end{proof}

From now on, we assume that $(v_1, \dots, v_n, w_1, \dots, w_n, e_{2n+1})$ is the Heisenberg basis provided by Proposition~\ref{Heisbasis}.
We can identify $\V$ with $\mathbb{R}^k$ and $\W$ with $\mathbb{R}^{2n+1-k}$ through the following diffeomorphisms
$$i_{\V}:  \V \to \mathbb{R}^k,  \ i_{\V}\left(\sum_{i=1}^k x_i v_i\right)=(x_1, \dots, x_k),$$
$$ i_{\W}: \W \to \mathbb{R}^{2n+1-k} , $$
$$ i_{\W}\left(z e_{2n+1} + \sum_{i=k+1}^{n} (x_i v_i+y_iw_i)+ \sum_{i=1}^k \eta_i v_i \right)= ( x_{k+1}, \dots, x_n, \eta_1, \dots, \eta_k, y_{k+1}, \dots, y_n, z).$$

We identify any function from an open subset $U \subset \W$, $\phi: U  \to \V$ with the corresponding function from an open subset $\widetilde{U} \subset \R^{2n+1-k}$, $\widetilde{\phi}: \widetilde{U}  \to \mathbb{R}^k$:
$$\widetilde{\phi}(w)=i_{\V}(\phi(i_{\W}^{-1}(w))) \ \ \ \forall w \in \widetilde{U}=i_{\W}(U) \subset \mathbb{R}^{2n+1-k}.$$
Any h-homomorphism $L: \W \to \V$ can be identified with 
the linear map $\widetilde{L}: \R^{2n+1-k} \to \R^k$ with respect
to the fixed basis. 
So it can be identified with a $k \times (2n-k)$ 
matrix $M_L$ with real coefficients such that
$$
L(w)=M_L \pi(w)^T,
$$
for every $w \in \R^{2n+1-k}$, where $\pi: \R^{2n+1-k} \to \R^{2n-k}$ is the canonical projection on the first $2n-k$ components.

If $U \subset \W$ is an open set and $\phi: U \to \V$ is intrinsic differentiable at a point $w \in U$, we denote by $D^{\phi} \phi(w)$ the matrix associated to $d \phi_{w}$ and we call it \emph{intrinsic Jacobian matrix} of $\phi$ at $w$.
If $\mathcal{U} \subset \R^{2n+1-k}$ is an open set and $\psi=(\psi_1, \dots, \psi_k): \mathcal{U} \to \R^k$ is a function, we define the family of $2n-k$ vector fields:
$$ W^{\psi}_j= \begin{cases}
(i_{\W})_*(X_{j+k}) \ \ \ \ \ j= 1, \dots, n-k\\
\nabla^{\psi_{j-n+k}}=\partial_{\eta_{j-n+k}}+ \psi_{j-n+k} \partial_{z} \ \ \ \ \ j= n-k+1, \dots, n \ \\\
(i_{\W})_*(Y_{j+k})\ \ \ \ \ j= n+1, \dots, 2n-k.\\
\end{cases}
$$
\begin{deff}[Intrinsic derivatives]
Let $U \subset \W$ be an open set and let $\bar{w}$ be a point of $U$. Let $\phi: U \to \V $ be a continuous function. For each
$j= 1, \ldots, 2n-k$, we say that $\phi$ has {\em $\partial^{\phi_j}$-derivative at $\bar{w}$} if and only if there exists 
$(\alpha_{1,j}, \dots, \alpha_{k,j})\in\mathbb{R}^k$ such that for all integral curves $\gamma^j: (-\delta, \delta) \to U$ of $W^{\phi}_j$ with $\gamma^j(0)=\bar{w}$ the following limit
$$  \lim_{t \to 0} \frac{\phi (\gamma^j(t))-\phi(\bar{w})}{t}$$ 
exists and it is equal to 
$(\alpha_{1,j},\ldots,\alpha_{k,j}).$
For all $j= 1, \dots, 2n-k$ we denote it by 
\[
\partial^{\phi_j} \phi (\bar{w})  = 
\begin{pmatrix}
\partial^{\phi_j} \phi_1(\bar{w})\\
\vdots\\
\partial^{\phi_j} \phi_k(\bar{w})
\end{pmatrix}=
\begin{pmatrix}
\alpha_{1,j}\\
\vdots\\
\alpha_{k,j}.
\end{pmatrix}.
\]
\end{deff}
The existence of continuous intrinsic derivatives 
actually characterizes the uniform intrinsic differentiability, 
\cite[Theorem 5.7]{Intsurfaces}.

\begin{deff}
\label{intjacobian}
Let $U \subset \W$ be an open set. Let $\phi: U \to \V$ be an intrinsic differentiable function at $\bar{w} \in U$.
We define the \emph{intrinsic Jacobian of $\phi$ at $\bar{w}$} as
$$ J^{\phi} \phi (\bar{w})=  \sqrt{ 1 + \sum_{\ell=1}^k\sum_{I \in \mathcal{I}_{\ell}}  (M^{\phi}_I(\bar{w}))^2 },$$
where we have defined $\mathcal{I}_{\ell}$ as the set of multiindexes
$$ \{ (i_1, \dots, i_{\ell},j_1, \dots, j_{\ell})) \in \mathbb{N}^{2l} : 1 \leq i_1 < i_2 < \dots < i_{\ell} \leq 2n-k, \ 1 \leq  j_1 < j_2 \dots < j_{\ell} \leq k \}. $$
We have also introduced the minors
\begin{equation*}
 M^{\phi}_I(\bar{w}) =  \mathrm{det} \begin{pmatrix} 
\partial^{\phi_{i_1}} \phi_{j_1}(\bar{w}) & \dots & \partial^{\phi_{i_{\ell}}} \phi_{j_1}(\bar{w}) \\
\dots & \dots & \dots \\
\partial^{\phi_{i_1}} \phi_{j_{\ell}}(\bar{w}) & \dots & \partial^{\phi_{i_{\ell}}} \phi_{j_{\ell}}(\bar{w}) \\
\end{pmatrix}.
\end{equation*}

\end{deff}

\subsection{Measures and area formulas}

If $\mathbb{H}^n$ is endowed with a homogeneous distance $d$, we denote by $\mathbb{B}(x,r)=  \{ y \in \mathbb{H}^n : d(x,y) \leq r \}$ and for $S \subset \mathbb{H}^n$, 
$$\mathrm{diam}(S)= \sup \{ d(x,y) : x,y \in S \}.$$ 
Notice that $ \mathrm{diam}(\mathbb{B}(x,r))=2r$ for all $x \in \mathbb{H}^n$ and $r >0$.
\begin{deff}[Carath\'eodory's construction]\label{d:caratheodory}
Let $\mathcal{F} \subset \mathcal{P}(\mathbb{H}^n) $ be a non-empty family of closed subsets of $\mathbb{H}^n$, equipped with a homogeneous distance $d$. Let be $\alpha>0$. If $\delta>0$, and $A \subset \mathbb{H}^n$, we define
\begin{equation}
\label{caratheodory}
\phi_{\delta}^{\alpha}(A)= \inf \left\{ \sum_{j=0}^{\infty} c_{\alpha} \ \mathrm{diam}(B_j)^{\alpha}: A \subset \bigcup_{j=0}^\infty B_j , \ \mathrm{diam}(B_j) \leq \delta, \ B_j \in \mathcal{F} \right\},
\end{equation}
If $\mathcal{F}$ coincides with the family $\mathcal{F}_b$ of closed balls
with respect to the distance $d$ and we choose $c_{\alpha}=2^{-\alpha}$
in (\ref{caratheodory}), then 
\begin{equation}\label{eq:Salpha}
\mathcal{S}^{\alpha}(A)= \sup_{\delta>0} \phi_{\delta}^{\alpha}(A)
\end{equation}
is the \emph{$\alpha$-spherical measure} of $A\subset\H^n$. 
\end{deff}
In the case $\mathcal{F}$ is the family of all
closed sets and $ k \in \{1, \dots, 2n+1 \}$, we define 
\[
c_k=\frac{\cL^k(\set{x\in\R^k: |x|\le 1})}{2^k}
\]
where $\mathcal{L}^k$ denotes the Lebesgue measure. Then the corresponding $k$-dimensional Hausdorff measure is given by
$$ \mathcal{H}^k_E (A)= \sup_{\delta>0} \phi_{\delta}^{k}(A)$$
where $\H^n$ is equipped with the Euclidean distance induced through the identification with
$\mathbb{R}^{2n+1}$.  These measures are Borel regular on subsets of $\mathbb{H}^n$.
For our purposes, it is useful to recall a less known Hausdorff-type measure, first introduced in \cite{ray}. Given $\alpha \in [0, \infty)$, $\delta \in (0, \infty)$, we define the $\alpha$-{\em dimensional centered Hausdorff measure} $\mathcal{C}^{\alpha}$ of a set $A \subset \mathbb{H}^n$ as 
$$ \mathcal{C}^{\alpha}(A)= \sup_{E \subset A} \mathcal{D}^{\alpha}(E)$$ where
$\mathcal{D}^{\alpha}(E)= \lim_{\delta \to 0+} \mathcal{C}^{\alpha}_{ \delta}(E)$, and, in turn, $\mathcal{C}^{\alpha}_{\delta}(E)=0$ if $E = \emptyset$ and if $E \neq \emptyset$
$$
C^{\alpha}_{ \delta}(E)= \inf \  \left\{ \  \sum_{i=0}^\infty  r_i^{\alpha} : E \subset \bigcup_{i=0}^\infty \mathbb{B}(x_i, r_i), \ x_i \in E, \ \text{diam}(\mathbb{B}(x_i, r_i)) \leq \delta  \right\}.
$$
\begin{deff}
Let $\alpha>0$, $x \in \mathbb{H}^n$ and $\mu$ be a Borel regular measure on $\mathbb{H}^n$. We define the \emph{upper $\alpha$-density of $\mu$ at $x $} as 
\begin{equation}\label{eq:ThetaUpperDensity}
\Theta^{*\alpha}(\mu,x)= \limsup_{r \to 0}\frac{\mu(\mathbb{B}(x,r))}{r^{\alpha}}. 
\end{equation}
\end{deff}

The previous definition and terminology follow 
\cite[2.10.19]{Federer}.

\begin{teo}[{\cite[Theorem 3.1]{FSSC8}}]
\label{abstractdiffcent}
Let $\alpha>0$ and let $\mu$ be a Borel regular measure on $\mathbb{H}^n$ such that there exists a countable open covering of $\mathbb{H}^n$, whose elements have $\mu$ finite measure. Let $B \subset A \subset \mathbb{H}^n$ be Borel sets. If $\mathcal{C}^{\alpha}(A) < \infty$ and $\mu \res A$ is absolutely continuous with respect to $\mathcal{C}^{\alpha} \res A$, then we have that $\Theta^{*\alpha}(\mu, \cdot)$ is a Borel function on $A$ and
$$ \mu(B)= \int_B \Theta^{*\alpha}(\mu,x) \ d  \mathcal{C}^{\alpha}(x).$$
\end{teo}

We introduce now a crucial definition of density.
\begin{deff}
\label{def:sphefederer}
Let $\mathcal{F}_b$ be the family of closed balls with positive radius in $\H^n$ endowed with a homogeneous distance $d$. Let $\alpha>0$, $x \in \mathbb{H}^n$ and $\mu$ be a Borel regular measure on $\mathbb{H}^n$. We call \emph{spherical $\alpha$-Federer density of $\mu$ at $x $} the real number 
$$
\theta^{\alpha}(\mu,x)= \inf_{\epsilon>0} \sup \left\{ \frac{2^\alpha\mu(\mathbb{B})}{\diam(\B)^{\alpha}} : x \in \mathbb{B} \in \mathcal{F}_b, \ r < \epsilon \right\}.
$$
\end{deff}

This density naturally appears in representing a Borel regular measure that is absolutely continuous with respect to the $\alpha$-dimensional spherical measure.

\begin{teo}[{\cite[Theorem~7.2]{Magnani2019}}]
\label{abstractdiff}
Let $\alpha>0$ and let $\mu$ be a Borel regular measure on $\mathbb{H}^n$ such that there exists a countable open covering of $\mathbb{H}^n$ whose elements have $\mu$ finite measure. If $B \subset A \subset \mathbb{H}^n$ are Borel sets, then $\theta^{\alpha}(\mu, \cdot)$ is a Borel function on $A$. If in addition $\mathcal{S}^{\alpha}(A) < \infty$ and $\mu \res A$ is absolutely continuous with respect to $\mathcal{S}^{\alpha} \res A$, then 
$$ \mu(B)= \int_B \theta^{\alpha}(\mu,x) \ d  \mathcal{S}^{\alpha}(x).$$
\end{teo}

\begin{deff}[Spherical factor]
\label{def:sphfactor}
Let $d$ be a homogeneous distance in $\H^n$. If $\Pi\subset\H^n$ is a linear subspace of topological dimension $p$, then the \emph{spherical factor} of $\Pi$ with respect to $d$ is 
$$ \beta_d (\Pi)= \max_{z \in \mathbb{B}(0,1)} \mathcal{H}^{p}_{E} (\Pi \cap \mathbb{B}(z,1)).$$
\end{deff}
When we deal with a homogeneous distance $d$ that preserves some symmetries, then the spherical factor can become a geometric constant. 
The following definition detects those homogeneous distances giving a constant spherical factor. It extends \cite[Definition~6.1]{Magnani2017} to higher codimension.

\begin{deff}
\label{def:vertsymm}
We refer to the fixed graded scalar product $\langle \cdot, \cdot \rangle$ on $\H^n$ and we assume that there exists a family $\mathcal{F} \subset O(H_1)$ of isometries such that for any couple of $(p-1)$-dimensional subspaces $S_1, S_2 \subset H_1$, there exists $L \in \mathcal{F}$ that satisfies the condition $$L(S_1)=S_2.$$
Let $d$ be a homogeneous distance on $\mathbb{H}^n$
and let $p=1,\ldots,2n$. 
We say that $d$ is {\em $p$-vertically symmetric} 
if $p=1$ or $p\ge2$ and the following conditions hold.

Taking into account that $H_1$ and $H_2$ are orthogonal, we introduce the class of isometries 
$$\mathcal{O} = \{ T \in O(\mathbb{H}^n) :  T|_{H_2}=\text{Id}|_{H_2},  \  T|_{H_1} \in \mathcal{F} \}.$$
We also assume the following:
\begin{itemize}
\item $\pi_{H_1}(\mathbb{B}(0,1)) = \mathbb{B}(0,1) \cap H_1 = \{h \in H_1 : \theta(|\pi_{H_1}(h)|)\le  r_0\}$ for some monotone non-decreasing function $\theta : [0,+ \infty) \to [0,+ \infty)$ and $r_0 > 0$,
\item $ T(\mathbb{B}(0,1)) = \mathbb{B}(0,1)$ for all $T \in \mathcal{O}$. 
\end{itemize}
\end{deff}

More information on $p$-vertically symmetric distances can be found in \cite{Mag21prRotSym}. For instance, the sub-Riemannian distance in the Heisenberg group is vertically symmetric. Vertically symmetric distances were already introduced in \cite{Magnani2017}.

The next theorem specializes \cite[Theorem~1.1]{Magnani2020} to Heisenberg groups.

\begin{teo}\label{fattorevertical}
If $p=1,\ldots,2n+1$ and $d$ is a $p$-vertically symmetric distance on $\mathbb{H}^n$, then the spherical factor
$\beta_d(\mathbb{W})$ is constant on every $p$-dimensional vertical subgroup $\mathbb{W}\subset \mathbb{H}^n$. 
\end{teo}

The previous theorem motivates the following definition.

\begin{deff}[Notation for constant spherical factors]\label{not:rotpn}
Let $\cN_p$ be the family of all $p$-dimensional 
vertical subgroups of $\H^n$.
We consider a homogeneous distance $d$.
We assume that the spherical factor 
$\beta_d(S)$ remains constant as $S$ varies in $\cN_p$
(this means that $d$ is {\em rotationally symmetric} with respect to $\cN_p$).
We denote the constant spherical factor by $\omega_d(p)$, 
without indicating the class $\cN_p$.
\end{deff}

\begin{deff}[{\cite[Definition 8.5]{Magnani2019}}]\label{d:multiradial}
Let $d$ be a homogeneous distance on $\mathbb{H}^n$.
We say that $d$ is \emph{multiradial} if there exists a function $\theta : [0, + \infty)^2 \to  [0,+ \infty)$, which is continuous and monotone non-decreasing on each single variable, with $$ d(x,0) = \theta(|\pi_{H_1}(x)|, |\pi_{H_2}(x)|).$$
The function $\theta$ is also assumed to be coercive in the sense that $\theta(x) \to +\infty$ as $|x| \to + \infty$. 
\end{deff}

\begin{prop}
\label{multiimpliesvert}
If $d: \mathbb{H}^n \times \mathbb{H}^n \to [0, \infty)$ is multiradial, then it is also $p$-vertically symmetric for every $p=1, \dots, 2n+1$.
\end{prop}

A more general statement can be found in \cite{Mag21prRotSym}. One may also check that
both $d_{\infty}$ and the Cygan--Kor\'anyi distance are multiradial.
One can find conditions under which the spherical factor has a simpler representation. The next theorem is established in \cite[Theorem~1.4]{Magnani2020}.

\begin{teo}\label{t:pallaconvex}
If $p=1,\ldots,2n+1$ and $d$ is a homogeneous distance in $\mathbb{H}^n$ whose unit ball $\mathbb{B}(0,1)$ is convex, then for every $p$-dimensional vertical subgroup $\mathbb{W}$ we have
$$\beta_d(\mathbb{W})= \mathcal{H}^{p}_E(\mathbb{W} \cap \mathbb{B}(0,1)).$$
\end{teo}

\section{Upper blow-up of low codimensional $\H$-regular surfaces}

In this section we prove the main technical tool of the paper, 
that is the equality between spherical Federer density 
and spherical factor, established in Theorem~\ref{Maintheorem}.
The next lemma will be important for the proof of
our technical result. It gives a formula of how the 
area transforms under a suitable linear isomorphism 
between two vertical groups.

\begin{lem}\label{l:MVPsubgroups}
We consider two vertical subgroups $\M$, $\W$ of $\H^n$ and a $k$-dimensional horizontal subgroup $\V\subset \mathbb{H}^n$ such that 
\[
\H^n=\M\rtimes \V=\W\rtimes\V.
\]
We introduce the multivectors 
\[
V=v_1\wedge \cdots\wedge v_k,\quad N=w_1\wedge \cdots\wedge w_{2n-k}\wedge e_{2n+1},\quad
M=m_1\wedge\cdots\wedge m_{2n-k}\wedge e_{2n+1},
\]
where $(v_1,\ldots, v_k)$, $(w_1, \ldots w_{2n-k},e_{2n+1})$ and $(m_1,\ldots, m_{2n-k},e_{2n+1})$
are orthonormal bases of $\V$, $\W$ and $\M$, respectively.
Then for every Borel set $B\subset \M$, we have
$$ (\pi_{\M,\W}^{\mathbb{M,V}})_\sharp \mathcal{H}_E^{2n+1-k} (B)=
\mathcal{H}_E^{2n+1-k}(\pi_{\W,\mathbb{M}}^{\mathbb{W,V}}(B))= \frac{\Vert V \wedge M \Vert_g}{\Vert V \wedge N \Vert_g} \mathcal{H}^{2n+1-k}_E(B),
$$
where the projections $\pi^{\M,\V}_{\M,\W}$ and $\pi^{\W,\V}_{\W,\M}$ have been introduced in Definition~\ref{d:projWVM}. The norms of $V\wedge M$ and $V\wedge N$ are taken with respect to the Hilbert structure of $\Lambda_{2n+1}(\H^n)$ induced by our scalar product on $\H^n$.

\end{lem}
\begin{proof}
It is clearly not restrictive to relabel the bases of $\M$ and $\W$ as 
$w_{k+1},\ldots,w_{2n},e_{2n+1}$ and $m_{k+1},\ldots,m_{2n},e_{2n+1}$.
We define the isomorphisms $i_\W : \mathbb{W} \to \mathbb{R}^{2n+1-k}$, 
\[
i_\W\pa{x_{2n+1} e_{2n+1}+\sum_{i=k+1}^{2n} x_iw_i }=(x_{k+1},\ldots, x_{2n+1})
\]
and $i_\M: \M\to \mathbb{R}^{2n+1-k}$, 
\[
i_\M\pa{x_{2n+1} e_{2n+1}+\sum_{i=k+1}^{2n} x_i m_i}=(x_{k+1},\ldots, x_{2n+1})
\]
and $i_{\V} : \V \to \R^k$
\[
i_\V\pa{\sum_{i=i}^{k} x_i v_i}=(x_{1},\ldots, x_{k}).
\]

We introduce $\Psi_1:\mathbb{R}^{2n+1}\to\mathbb{H}^n$,

\begin{equation}\label{eq:Psi1}
\Psi_1(x_1, \dots , x_{2n+1})=	\Big(x_{2n+1} e_{2n+1}+\sum_{i=k+1}^{2n} x_i w_i\Big)\Big(\sum_{j=1}^k x_i v_i\Big).
\eeq
We now notice that $J\Psi_1(x)= \Vert V \wedge N \Vert_g$
for every  $x=(x_1, \ldots , x_{2n+1})\in\R^{2n+1}$.
It suffices to observe that 
\[
J\Psi_1=\|\der_{x_1}\Psi_1\wedge \cdots \der_{x_{2n+1}}\Psi_{2n+1}\|_g
\]
and use the explicit form of \eqref{eq:Psi1}. We define another map $\Psi_2:\mathbb{R}^{2n+1}\to \mathbb{H}^n$,
\[
\Psi_2(x_1, \dots , x_{2n+1})=\Big(x_{2n+1} e_{2n+1}+\sum_{i=k+1}^{2n} x_i m_i\Big)\Big(\sum_{j=1}^k x_i v_i\Big),
\]
and we observe in the same way that $J \Psi_2(x)= \Vert V \wedge M \Vert_g.$ 
We introduce the embedding $q:\mathbb{R}^{2n+1-k} \to \mathbb{R}^{2n+1}$,
\[
q(x_1, \dots, x_{2n+1-k})= (0, \dots, 0, x_1, \dots, x_{2n+1-k})
\]
and the projection $p: \mathbb{R}^{2n+1} \to \mathbb{R}^{2n+1-k}$,
\[
p(x_1, \dots, x_{2n+1})=(x_{k+1}, \dots, x_{2n+1}).
\]
For every $z\in\H^n$, we observe that
\[
\Psi_1^{-1}(z)=\pa{i_\V\circ\pi_\V(z),i_\W\circ \pi_\W(z)}.
\]
It follows that 
\[
i_\W^{-1}\circ p\circ \Psi_1^{-1}=\pi_\W.
\]
If we take any $m\in\M$, then
\begin{equation}\label{eq:piWMm}
\begin{split}
\pi_\W(m)&=i_\W^{-1}\circ p\circ \Psi_1^{-1}\circ\Psi_2\circ\Psi_2^{-1}(m) \\
&=i_\W^{-1}\circ p\circ \Psi_1^{-1}\circ\Psi_2\circ q\circ i_\M(m)\\
&=\pi^{\W,\V}_{\W,\M}(m).
\end{split}
\end{equation}
The second equality follows by the identity 
\[
\Psi_2^{-1}=(i_\V\circ\pi_\V,i_\M\circ\pi_\M),
\]
hence $\Psi_2^{-1}(m)=(0,i_\M(m))$ for all $m\in\M$.
We notice that $\Psi_1^{-1}\circ\Psi_2$ is a polynomial diffeomorphism, whose Jacobian matrix at $x$
has the following form 
\[
\pa{\begin{array}{ccc} 
I & R_1 & 0 \\
0 & R_2 & 0 \\
\ell_1(x) & \ell_2(x) & 1
\end{array}}\in\R^{(2n+1)\times(2n+1)},
\]
where $I\in\R^{k\times k}$, $R_1\in \R^{k\times (2n-k)}$, $R_2 \in \R^{(2n+1-k) \times (2n+1-k)}$  and
the functions 
\[
\ell_1:\R^{2n+1}\to\R^k\quad \text{and}\quad \ell_2:\R^{2n+1}\to\R^{2n-k}
\]
are affine. From definition of $q:\R^{2n+1-k}\to \R^{2n+1}$ and of $p:\R^{2n+1}\to\R^{2n+1-k}$,
by explicit computation, it follows that
\beq\label{eq:formula}
J(\Psi_1^{-1}\circ\Psi_2)(q(y))=|\det R_2|=J(p\circ \Psi_1^{-1}\circ \Psi_2\circ q)(y),
\eeq
for every $y\in\R^{2n+1-k}$.
As a consequence, taking into account \eqref{eq:piWMm}, \eqref{eq:formula} and 
\[
\frac{\|V\wedge M\|_g}{\|V\wedge N\|_g}=J(\Psi_1^{-1}\circ\Psi_2),
\]
the following equalities hold
\begin{equation*}
\begin{split}
\mathcal{H}^{2n+1-k}_E(B) &= \mathcal{L}^{2n+1-k}(i_{\M}(B)) \\
&= \frac{\Vert V  \wedge N \Vert_g}{\Vert V \wedge M \Vert_g}  \mathcal{L}^{2n+1-k}((p\circ \Psi_1^{-1}\circ \Psi_2\circ q)(i_{\M}(B)))\\
&= \frac{\Vert V  \wedge N \Vert_g}{\Vert V \wedge M \Vert_g} \mathcal{H}^{2n+1-k}_E((i_{\W}^{-1}\circ p\circ \Psi_1^{-1}\circ \Psi_2\circ q\circ i_{\M})(B))\\
&= \frac{\Vert V  \wedge N \Vert_g}{\Vert V \wedge M \Vert_g} \mathcal{H}^{2n+1-k}_E(\pi_{\W, \M}^{\W, \V}(B))
\end{split}
\end{equation*}
for every Borel set $B\subset\M$.
\end{proof}
We are now in the position to present our main technical result.

\begin{teo}[Upper blow-up]	\label{Maintheorem}
We consider a semidirect factorization $\H^n=\W\rtimes \V$,
an open set $\Omega\subset\H^n$, a function $f\in C^1_h(\Omega,\R^k)$
and a homogeneous distance $d$. 
We fix $x_0\in\Omega$ and the level set 
$\Sigma=f^{-1}(f(x_0))$, assuming that $J_\V f(x)>0$ for all $x\in\Sigma$. 
We choose the orthonormal bases $(v_1, \dots v_k)$ of $\V$
and $(w_{k+1}, \dots, w_{2n}, e_{2n+1})$ of $\W$,
setting $V= v_1 \wedge \dots \wedge v_k $ and
$N= w_{k+1} \wedge \dots \wedge w_{2n} \wedge e_{2n+1}$. 
Then the following conditions hold.
\begin{enumerate}
\item $\Sigma$ is a parametrized $\H$-regular surface 
with respect to $(\W,\V)$.
\item 
If we denote by $\phi:U\to \V$ the parametrization of $\Sigma$ and 
introduce the measure
	\begin{equation}\label{eq:areamu}
		\mu(B)= \Vert V \wedge N \Vert_g \int_{\Phi^{-1}(B)} \frac{J_Hf(\Phi(n))}{J_\V f(\Phi(n))}  \  d \mathcal{H}_E^{2n+1-k} (n)
	\end{equation}
for every Borel set $B \subset \mathbb{H}^n$, where $\Phi(n)=n\phi(n)$, 
	then for every $x \in \Sigma$ we have 
	\begin{equation}\label{eq:equalityBetaSigma}
	\theta^{2n+2-k}( \mu, x)= \ \beta_{d}(\Tan(\Sigma,x)).
	\end{equation}
\end{enumerate}
\end{teo}

\begin{proof}
The first part of our claim is a consequence of 
Proposition~\ref{prop:parametrJV}.
Then our thesis follows once we have proved \eqref{eq:equalityBetaSigma}.
By formula \eqref{eq:areamu}, for any $y \in \Omega$, taking $t>0$ sufficiently small, we can write
\begin{equation} 
\mu(\mathbb{B}(y,t))= 
\Vert V \wedge N \Vert_g \int_{\Phi^{-1}(\mathbb{B}(y,t))} \frac{J_Hf(\Phi(n))}{J_\V f(\Phi(n))}  \  d \mathcal{H}_{E}^{2n+1-k} (n).
\end{equation}
We denote by $\zeta \in U$ the element such that 
\[
x= \Phi(\zeta)= \zeta \phi(\zeta).
\]
We now perform the change of variables
\[
n= \sigma_x(\Lambda_t(\eta))=x(\Lambda_t \eta)(\pi_\V(x))^{-1}=x(\Lambda_t \eta)(\phi(\zeta))^{-1},
\]
where $\Lambda_t= \delta_t|_{\mathbb{W}}$. The Jacobian of $\Lambda_t$ is $ t^{2n+2-k}$.
It is well known that $\sigma_x$ has unit Jacobian (see for instance {\cite[Lemma 2.20]{IntLipgraphs}}).
Setting $\alpha(x)=J_Hf(x)/J_\V f(x)$, we obtain that
\[
\frac{\mu(\mathbb{B}(y,t))}{t^{2n+2-k}} =
\Vert V \wedge N \Vert_g  \int_{\Lambda_{1/t}(\sigma_x^{-1}(\Phi^{-1}(\mathbb{B}(y,t))))} 
(\alpha\circ \Phi)(\sigma_x (\Lambda_t(\eta))))  \  d \mathcal{H}^{2n+1-k}_E (\eta).
\]
By the general definition of spherical Federer density we obtain that
\begin{equation*}
\begin{split}
\theta^{2n+2-k}(\mu ,x)& = \inf_{r>0} \sup_{\substack{y \in \mathbb{B}(x,t)\\ 0<t<r}} \frac{\mu (\mathbb{B}(y,t))}{t^{2n+2-k}} \\
&=  \inf_{r>0} \sup_{\substack{ y \in \mathbb{B}(x,t) \\ 0<t<r}} \Vert V \wedge N \Vert_g \ \int_{\Lambda_{1/t}(\sigma_x^{-1}(\Phi^{-1}(\mathbb{B}(y,t))))} 
(\alpha\circ \Phi) (\sigma_x(\Lambda_t(\eta)) ) \  d \mathcal{H}^{2n+1-k}_E (\eta).
\end{split}
\end{equation*}
There exists $R_0>0$ such that for $t>0$ and $y \in \mathbb{B}(x,t)$ 
we have the following inclusion
\begin{eqnarray}\label{insieme}
\Lambda_{1/t}(\sigma_x^{-1}(\Phi^{-1}(\mathbb{B}(y,t)))) \subset \mathbb{B}_{\mathbb{W}}(0,R_0),
\end{eqnarray}
where the translated function $\phi_{x^{-1}}$ is defined according to formula \eqref{eq:translatedf} and
we have set
\[
\mathbb{B}_{\mathbb{W}}(0,R_0)=\mathbb{B}(0,R_0)\cap \mathbb{W}.
\]
To see \eqref{insieme}, we write more explicitly $\Lambda_{1/t}(\sigma_x^{-1}(\Phi^{-1}(\mathbb{B}(y,t))))$,
that is 
\[
\left\{ \eta \in \Lambda_{1/t}(\sigma_x^{-1}(U)) : \left\|y^{-1}x (\Lambda_t \eta) \phi(\zeta)^{-1} \phi(x(\Lambda_t \eta) \phi(\zeta)^{-1}) \right\|\le t \right\}. 
\]
It can be written as follows
\[
\left\{ \eta \in \Lambda_{1/t}(\sigma_x^{-1}(U)) : \left\|(\delta_{1/t}(y^{-1}  x ))  \eta \left(\frac{\phi(\zeta)^{-1} \phi(x (\Lambda_t \eta) \phi(\zeta)^{-1})}{t}\right) \right\|\le 1 \right\}. 
\]
According to \eqref{eq:translatedf}, the translated function of $\phi$ at $x^{-1}$ is 
\[
\phi_{x^{-1}}(\eta)=\pi_\V(x^{-1})\phi(x\eta\pi_\V(x^{-1}))=
\phi(\zeta)^{-1}\phi(x\eta\phi(\zeta)^{-1}).
\]
We finally get 
\beq\label{eq:setLambda}
\Lambda_{1/t}(\sigma_x^{-1}(\Phi^{-1}(\mathbb{B}(y,t)))) = \left\{ \eta \in \Lambda_{1/t}(\sigma_x^{-1}(U)) : 
\left\|(\delta_{1/t}(y^{-1}  x ))  \eta \left(\frac{\phi_{x^{-1}}(\Lambda_t \eta)}{t}\right)\right\|\le1 \right\},
\eeq
hence for $\eta \in \Lambda_{1/t}(\sigma_x^{-1}(\Phi^{-1}(\mathbb{B}(y,t)))$, taking into account 
the previous equality, we have established that
$$ \eta\left(\frac{\phi_{x^{-1}}(\Lambda_t \eta)}{t}\right) \in \mathbb{B}(0,2).$$
From the estimate \eqref{eq:estimnormsub}, we know that
$$c_0 \left( \Vert \eta \Vert + \left\Vert   \frac{\phi_{x^{-1}}(\Lambda_t\eta)}{t}\right\Vert \right)\leq 
\left\Vert \eta \left(\frac{\phi_{x^{-1}}(\Lambda_t\eta)}{t}\right) \right\Vert \leq 2,$$
hence the inclusion \eqref{insieme} holds with $R_0=2/c_0$. As a consequence, we have that $$ \theta^{2n+2-k}(\mu ,x) < \infty.$$
There exist a positive sequence $t_p$ converging to zero and $y_p \in \mathbb{B}(x,t_p)$ such that
\[
\Vert V \wedge N \Vert_g \int_{\Lambda_{1/t_p}(\sigma_x^{-1}(\Phi^{-1}(\mathbb{B}(y_p,t_p))))} \frac{J_Hf(\Phi (\sigma_x(\Lambda_{t_p} (\eta) )}{J_{\V}f(\Phi(\sigma_x(\Lambda_{t_p} (\eta) ))}\ d  \mathcal{H}^{2n+1-k}_E(\eta)\to \theta^{2n+2-k}(\mu,x)
\]
as $p\to\infty$. Up to extracting a subsequence, since $y_p \in \mathbb{B}(x,t_p)$ for every $p$, there exists $z \in \mathbb{B}(0,1)$ such that
$$ \lim_{p \to \infty} \delta_{1/t_p}(x^{-1}  y_p) = z.$$
For the sake of simplicity, we use the notation 
\[
\M_x=\ker Df(x).
\]
Using the projection introduced in Definition~\ref{d:projWVM}, we set
\[
S_z=\pi_{\W,\M_x}^{\W,\V}(\M_x\cap \mathbb{B}(z,1))\subset\W.
\]
\textbf{Claim 1:} For each $\omega \in \mathbb{W}\sm S_z$, there exists
$$
\lim_{p \to \infty} \bu_{\Lambda_{1/t_p}(\sigma_x^{-1}(\Phi^{-1}(\mathbb{B}(y_p,t_p)))} (\omega)=0.
$$
By contradiction, if we had a subsequence of the integers $p$ such that 
\begin{equation*}
(\delta_{1/t_p}(y_p^{-1}  x )) \omega \pa{\frac{\phi_{x^{-1}}(\Lambda_{t_p} \omega)}{t}}\in \mathbb{B}(0,1),
\end{equation*}
then by a slight abuse of notation, we could still call $t_p$ the sequence such that
\beq\label{eq:inclusion}
(\delta_{1/t_p}(y_p^{-1}  x ))  \omega    d \phi_{\zeta}(\omega)     \pa{\frac{ (d \phi_{\zeta}(\Lambda_{t_p} \omega))^{-1} \phi_{x^{-1}}(\Lambda_{t_p} \omega)}{t_p}}  \in \mathbb{B}(0,1)
\eeq
for all $p$, where we have used the homogeneity of the intrinsic differential $d\phi_\zeta$ of $\phi$,
see Definition~\ref{d:intrinsicDiff} for the notion of intrinsic differential.
Indeed, by Theorem \ref{t:ArenaSerapioni}, the function $\phi$ is in particular
intrinsic differentiable at $\zeta$.
Due to the intrinsic differentiability, taking into account \eqref{eq:inclusion} as $p\to\infty$,
it follows that
$$
\omega d \phi_{\zeta}(\omega) \in \mathbb{B}(z, 1).
$$
It is now interesting to observe that the chain rule of Theorem~\ref{t:chain} yields
\beq\label{eq:graphChain}
\mathrm{graph}(d \phi_{\zeta})= \ker(Df(x)))=\M_x.
\eeq
As a consequence, $ \omega d \phi_{\zeta}(\omega) \in \mathbb{B}(z,1) \cap \M_x$
and then 
\beq\label{eq:piS_z}
\omega=\pi_{\W,\M_x}^{\W,\V}(\omega d \phi_{\zeta}(\omega))\in \pi_{\W,\M_x}^{\W,\V}(\M_x \cap \mathbb{B}(z,1))=S_z,
\eeq
that is not possible by our assumption. This concludes the proof of Claim 1.

Now we introduce the density function
$$
\alpha (t, \eta)=\frac{J_Hf(\Phi( \sigma_x(\Lambda_t(\eta)))}{J_{\V}f(\Phi( \sigma_x(\Lambda_t(\eta)))}
$$
to write 
$$
\Vert V \wedge N \Vert_g \int_{\Lambda_{1/t_p}( \sigma_x^{-1}(\Phi^{-1}(\mathbb{B}(y_p,t_p))))} \alpha(t_p, \eta) \ d   \mathcal{H}^{2n+1-k}_E (\eta) = I_p+J_p.
$$
The sequence $I_p$, defined in the following equality, satisfies the estimate
\[
\begin{split}
I_p&= \Vert V \wedge N \Vert_g \  \int_{S_z \cap \Lambda_{1/t_p}(\sigma_x^{-1}(\Phi^{-1}(\mathbb{B}(y_p,t_p))))} \alpha(t_p, \eta) \ d   \mathcal{H}^{2n+1-k}_E (\eta) \\
&\leq \Vert V \wedge N \Vert_g \int_{S_z } \alpha(t_p, \eta) \ d   \mathcal{H}^{2n+1-k}_E (\eta).
\end{split}
\]
Analogously for $J_p$, we find  
\begin{equation*}
\begin{split}
J_p &= \Vert V \wedge N \Vert_g \ \int_{\Lambda_{1/t_p}(\sigma_x^{-1}(\Phi^{-1}(\mathbb{B}(y_p,t_p)))) \sm S_z} \ \alpha(t_p, \eta) \ d   \mathcal{H}^{2n+1-k}_E (\eta)\\
& \leq  \Vert V \wedge N  \Vert_g \int_{\mathbb{B}_{\mathbb{W}}(0,R_0) \sm S_z} \bu_{\Lambda_{1/t_p}(\sigma_x^{-1}((\Phi^{-1}(\mathbb{B}(y_p,t_p))))}(\eta) \ \alpha(t_p, \eta) \ d   \mathcal{H}^{2n+1-k}_E (\eta).
\end{split}
\end{equation*}
Claim 1 joined with the dominated convergence theorem prove that $J_p\to0$ as $p\to \infty$,
hence $I_p\to\theta^{2n+2-k}(\mu,x)$. To study the asymptotic behavior of $I_p$, we
first observe that
\[
\alpha(t_p,\eta) \to \frac{ J_Hf(x)}{J_{\V}f(x)}=c(x) 
\]
as $p \to \infty$. 
It follows that  
\beq\label{eq:theta}
\theta^{2n+2-k}(\mu,x)=\lim_{p\to\infty}I_p \leq  \Vert V \wedge N \Vert_g \ c(x) \  \mathcal{H}^{2n+1-k}_E(S_z).
\eeq

\textbf{Claim 2}. We set $\M_x=\ker(Df(x))$ and consider $N_x= m_{k+1} \wedge \dots \wedge m_{2n} \wedge e_{2n+1}$ such that $ ( m_{k+1}, \dots, m_{2n}, e_{2n+1} )$ is an orthonormal basis of $\M_x$. We have that
\beq\label{eq:c(x)}
c(x)=  \frac{ J_Hf(x)}{J_{\V}f(x)} = \frac{1}{\Vert V \wedge N_x \Vert_g}.
\eeq
Since $\spn \{ \nabla_{H} f_1(x), \dots, \nabla_{H}f_k(x) \}$
is orthogonal to $\M_x$, it is a standard fact that 
\begin{equation}
\label{ker2}
 m_{k+1} \wedge \dots \wedge m_{2n} \wedge e_{2n+1} = \ast ( \nabla_{H} f_1(x) \wedge \dots \wedge \nabla_{H}f_k (x)) \lambda
\end{equation}
for some $\lambda \in \mathbb{R}$, see for instance {\cite[Lemma 5.1]{Magnani2008}}. Here we have defined the Hodge operator $*$ in $\mathbb{H}^n$ with respect to the fixed orientation $$\textbf{e}= e_1 \wedge \dots e_{2n} \wedge e_{2n+1} $$ and the fixed scalar product $\langle \cdot, \cdot \rangle$. Precisely, we are referring to an orthonormal Heisenberg basis $(e_1, \dots, e_{2n}, e_{2n+1})$, according to Sections~\ref{sect:symplectic} and \ref{sect:metric}. Therefore $ \ast \eta$ is the unique $(2n+1-k)$-vector such that 
\begin{equation}
\label{eq10}
\xi \wedge \ast \eta= \langle \xi, \eta \rangle  \ \textbf{e}
\end{equation}
 for all $k$-vectors $\xi$. Since the Hodge operator is an isometry, we get
\begin{equation}
\label{eq9}
 | \lambda |= \frac{1}{\Vert \nabla_{H}f_1 (x)\wedge \dots \nabla_{H}f_k(x) \Vert_g}.
\end{equation}
Due to (\ref{eq9}) and (\ref{eq10}), we have
\begin{equation*}
\begin{split}
\Vert V \wedge N_x \Vert_g &= | \lambda|  \Vert \Vert v_1 \wedge \dots \wedge v_k \wedge ( \ast ( \nabla_{H}f_1 (x)\wedge \dots \wedge \nabla_{H}f_k (x))) \Vert_g \\
&= \frac{\Vert \langle v_1 \wedge \dots \wedge v_k, \nabla_{H}f_1(x) \wedge \dots \wedge \nabla_{H}f_k(x) \rangle \textbf{e} \Vert_g}{ \Vert  \nabla_{H}f_1 (x)\wedge \dots \wedge \nabla_{H}f_k (x)\Vert_g  }\\
&= \frac{| \langle v_1 \wedge \dots \wedge v_k, \nabla_{H}f_1(x) \wedge \dots \wedge \nabla_{H}f_k(x) \rangle|}{ \Vert  \nabla_{H}f_1 (x)\wedge \dots \wedge \nabla_{H}f_k (x)\Vert_g  }\\
&= \frac{\Vert \nabla_{\V} f_1 (x)\wedge \dots \wedge \nabla_{\V} f_k(x) \Vert_g}{ \Vert  \nabla_{H}f_1(x) \wedge \dots \wedge \nabla_{H}f_k(x) \Vert_g  }\\
&=\frac{J_{\V}f(x)}{J_Hf(x)},
\end{split}
\end{equation*}
hence establishing Claim 2.

As a result, taking into account \eqref{eq:theta}, we have proved that
\beq
\theta^{2n+2-k}(\mu ,x)  \leq \frac{ \Vert V \wedge N \Vert_g}{\Vert V \wedge N_x \Vert_g} \ \mathcal{H}^{2n+1-k}_E(S_z).
\eeq
By Lemma~\ref{l:MVPsubgroups}, for $B=\M_x\cap\B(z,1)$, the following formula holds
\beq\label{eq:LemmaAlg} 
\mathcal{H}^{2n+1-k}_E(\pi_{\W, \M_x}^{\W, \V}(\M_x\cap\B(z,1)))= \frac{\Vert V \wedge N_x \Vert_g}{\Vert V \wedge N \Vert_g} \mathcal{H}^{2n+1-k}_E(\M_x\cap\B(z,1)).
\eeq
It follows that 
\beq\label{eq:firstIneq}
\theta^{2n+2-k}(\mu ,x) \leq\mathcal{H}^{2n+1-k}_E(\M_x \cap \mathbb{B}(z,1)) \leq  \beta_d(\M_x).
\eeq
To prove the opposite inequality, we follow the approach of \cite[Theorem 3.1]{Magnani2017}.
We choose $z_0 \in \mathbb{B}(0,1)$ such that
\begin{equation}\label{eq:beta_d}
\beta_d(\M_x)= \mathcal{H}^{2n+1-k}_E(\M_x \cap \mathbb{B}(z_0,1) )
\end{equation}
and consider a specific family of points $y_t^0= x \delta_t z_0 \in \mathbb{B}(x,t)$. For a fixed $\lambda>1$, we have
$$
\sup_{0<t<r} \frac{\mu(\mathbb{B}(y_t^0,\lambda t))}{(\lambda t)^{2n+2-k}} \leq \sup_{\substack{y \in \mathbb{B}(x,t), \\ 0 < t < \lambda r}} \frac{\mu(\mathbb{B}(y,t))}{t^{2n+2-k}}
$$
for every $r>0$ sufficiently small, therefore
\beq\label{eq:limsuptheta}
\limsup_{t \to 0^+} \frac{\mu(\mathbb{B}(y_t^0,\lambda t) )}{(\lambda t)^{2n+2-k}} \leq \theta^{2n+2-k}(\mu ,x).
\eeq
We introduce the set 
\begin{equation*}
\begin{split}
A^0_t &= \Lambda_{1/\lambda t}(\sigma_x^{-1}(\Phi^{-1}(\mathbb{B}(y_t^0, \lambda t)))\\
& = \left\{ \eta \in \Lambda_{1/\lambda t} (\sigma_{x}^{-1}(U)) :  \eta  \left( \frac{\phi_{{x}^{-1}}(\Lambda_{\lambda t} \eta)}{\lambda t} \right) \in \mathbb{B}(\delta_{1/ \lambda} z_0, 1)    \right\}.
\end{split}
\end{equation*}
The second equality can be deduced from \eqref{eq:setLambda}.
Then we can rewrite
\begin{equation}\label{eq:mu}
\begin{split}
\frac{\mu(\mathbb{B}(y_t^0,\lambda t))}{(\lambda t)^{2n+2-k}} &= \Vert V \wedge N \Vert_g \int_{A_t^0} \alpha( \lambda t, \eta) d \mathcal{H}^{2n+1-k}_E( \eta) \\
&= \frac{\Vert V \wedge N \Vert_g}{\lambda^{2n+2-k}} \int_{\delta_{\lambda} A_t^0} \alpha( \lambda t, \delta_{1/\lambda} \eta) d \mathcal{H}^{2n+1-k}_E(\eta)
\end{split}
\end{equation}
The domain of integration satisfies
$$ \delta_{\lambda} A_t^0= \left\{ \eta \in \Lambda_{1/t}(\sigma_{x}^{-1}(U)) : \eta \left( \frac{\phi_{{x}^{-1}}(\Lambda_{ t} \eta)}{ t} \right) \in \mathbb{B}(z_0, \lambda) \right\}.$$
Due to \eqref{insieme} and the definition of $A^0_t$, we get  
\[
\delta_{\lambda} A_t^0 \subset \mathbb{B}_\W(0,\lambda R_0).
\]
\textbf{Claim 3:} For every $\eta \in \pi_{\W, \M_x}^{\W, \V}( \M_x \cap B(z_0, \lambda))$, we have
\beq\label{eq:limA^0=1}
\lim_{t \to 0^+}\bu_{\delta_{\lambda} A_t^0} (\eta)=1.
\eeq	
The intrinsic differentiability of $\phi$ at $\zeta$ shows that
\[
\eta \left( \frac{\phi_{{x}^{-1}}(\Lambda_{ t} \eta)}{ t} \right)\to \eta d\phi_\zeta(\eta)\quad\text{as}\quad t\to0.
\]
Taking into account \eqref{eq:invFactor} and \eqref{eq:piS_z}, we get
\[
\pi^{\M_x,\V}_{\M_x,\W}(\eta)=\eta d\phi_\zeta(\eta),
\]
hence our assumption on $\eta$ can be written as follows
\[
d\pa{\eta d\phi_\zeta(\eta),z_0}<\lambda.
\]
We conclude that $\eta\in \delta_{\lambda} A_t^0$ for any $t>0$ sufficiently small, therefore the limit \eqref{eq:limA^0=1} holds and the proof of Claim 3 is complete.

By Fatou's lemma, taking into account \eqref{eq:limsuptheta} and \eqref{eq:mu}
we get
\[
\frac{\Vert V \wedge N \Vert_g}{\lambda^{2n+2-k}} 
\int_{\pi_{\W, \M_x}^{\W, \V}( \M_x \cap B(z_0, \lambda))} \liminf_{t\to0}\pa{\bu_{\delta_{\lambda} A_t^0}(\eta) \alpha( \lambda t, \delta_{1/\lambda} \eta)} d \mathcal{H}^{2n+1-k}_E(\eta)\le \theta^{2n+2-k}(\mu ,x).
\]
Claim 3 joined with \eqref{eq:c(x)} yield
\[
\frac1{\lambda^{2n+2-k}} \frac{\Vert V \wedge N \Vert_g}{\Vert V \wedge N_x \Vert_g}
\mathcal{H}^{2n+1-k}_E\pa{\pi_{\W, \M_x}^{\W, \V}( \M_x \cap\B(z_0,1))}\le \theta^{2n+2-k}(\mu ,x).
\]
Applying again \eqref{eq:LemmaAlg}, we obtain
\[
\frac1{\lambda^{2n+2-k}} \mathcal{H}^{2n+1-k}_E(\M_x\cap\B(z_0,1))\le \theta^{2n+2-k}(\mu ,x).
\]
Taking the limit as $\lambda\to1^+$, considering \eqref{eq:beta_d}
and taking into account Proposition~\ref{prop:Tan},
the proof of \eqref{eq:equalityBetaSigma} is complete.
\end{proof}

The computation of the upper density \eqref{eq:ThetaUpperDensity} is simpler than 
computing the spherical Federer density.
In a sense, we have less degrees of freedom,
since the center of the ball for this density is fixed.
As a byproduct of our approach,
the following theorem
can be achieved by some simplifications in 
the proof of Theorem~\ref{Maintheorem}, getting a ``centered blow-up''.

\begin{teo}
\label{Maintheoremcentrato}
In the assumptions of Theorem \ref{Maintheorem}, for every $x \in \Sigma$, we have
$$ \Theta^{*2n+2-k}( \mu, x) =  \mathcal{H}^{2n+1-k}_E(\Tan(\Sigma,x) \cap \mathbb{B}(0,1)),
$$
where the metric ball $\B(0,1)$ refers to the fixed 
homogeneous distance $d$.
\end{teo}

\section{Some special cases for the area formula}\label{sect:area}

In this section, we analyze some consequences of 
the upper blow-up (Theorem~\ref{Maintheorem}).
We consider two cases: when the factors $\W$ and $\V$
are orthogonal and when the metric unit ball 
of the homogeneous distance is convex.
In the first case the measure
$\mu$ can be represented by the intrinsic derivatives of the parametrization, according to Theorem~\ref{areaintder}.

\begin{proof}[Proof of Theorem~\ref{areaintder}]
Since $\W$ and $\V$ are orthogonal, by Proposition \ref{Heisbasis} we can fix a Heisenberg basis $  (v_1, \dots, v_k,  v_{k+1}, \dots,v_n, w_1, \dots w_{2n}, e_{2n+1}) $ such that $\V= \spn \{ v_1, \dots, v_k \} $ and 
$\W= \spn \{ v_{k+1}, \dots, v_n, w_i, \dots, w_n, e_{2n+1} \}$. Our claim follows by representing the measure $\mu$ in terms of the intrinsic partial derivatives of the parametrization $\phi$ of $\Sigma$, arguing as in the proof {\cite[Theorem 6.1]{Intsurfaces}}. For the reader's convenience we report the main points of the proof.

Taking into account Theorem \ref{t:ArenaSerapioni}, $\Sigma=\Phi(\Omega)$ is the graph of a uniformly intrinsic differentiable function $\phi$. Arguing as in the proof of \cite[Theorem 4.1]{DiDonatoArt} or \cite[Theorem 4.2]{Arena}, there exist an open set $\Omega'\subset\H^n$ and a function $g \in C^1_h(\Omega',\R^k)$ such that $\Sigma \subset g^{-1}(0) $ and for every $m \in U$ the following holds
\begin{equation}
\label{didonato}
Dg(\Phi(m))= \begin{bmatrix} \nabla_{H}g_1(\Phi(m)) \\
\dots \\
\nabla_{H}g_k(\Phi(m)) 
\end{bmatrix}= \begin{bmatrix} \mathbb{I}_{k}  &-D^{\phi}\phi(m) & 0\\
\end{bmatrix},
\end{equation}
where $0$ denotes the vanishing column in the previous matrix. By Theorem \ref{teo:areaformulagenerale}, for any Borel set $B \subset \Sigma$, 
$$\mu(B)= \int_B \beta_{d}(\Tan(\Sigma,x)) \ d \mathcal{S}^{2k+2-k}(x)=\int_{\Phi^{-1}(B)} \frac{J_{H}g(\Phi(n))}{J_{\V}g(\Phi(n))} d \mathcal{H}^{2n+1-k}_E(n).$$
Notice that $J_{\V}g(\Phi(m))=1$ for every $m \in U$. By Definition \ref{intjacobian}, taking into account the form of $Dg(\Phi(m))$ in (\ref{didonato}), the proof is achieved.
\end{proof}

Combining Theorem~\ref{abstractdiffcent} and Theorem~\ref{Maintheoremcentrato}, we also get the area formula for the centered Hausdorff measure. It is the analogous of Theorem~ \ref{teo:areaformulagenerale}, where
the spherical measure is replaced by the centered
Hausdorff measure.
For the distance $d_\infty$,
the following theorem coincides with \cite[Theorem 4.1]{Areaformula}.

\begin{teo}
\label{Areaformulacent}
In the assumptions of Theorem \ref{Maintheorem}, for any Borel set $B \subset \Sigma $ we have
\begin{equation}\label{eq:areaC}
\mu(B)= \int_B  \mathcal{H}^{2n+1-k}_E(\Tan(\Sigma,x) \cap \mathbb{B}(0,1)) \ d \mathcal{C}^{2k+2-k}(x),
\end{equation}
where the metric ball $\B(0,1)$ refers to the fixed homogeneous distance $d$.
\end{teo}
As a consequence, using the previous formula, along with Theorem~\ref{teo:areaformulagenerale} and Theorem~\ref{t:pallaconvex}, we can show the equality between spherical measure and centered Hausdorff measure.

\begin{teo}\label{coincidence}
Let $d$ be a homogeneous distance on $\mathbb{H}^n$ such that $\mathbb{B}(0,1)$ is convex.
Let $\Sigma$ be a parametrized $\H$-regular surface with respect
to $(\W,\V)$. Then for every $x \in \Sigma$ we obtain $\Theta^{*2n+2-k}(\mu,x)= \theta^{2n+2-k}(\mu,x)$ and
in particular
\beq
\mathcal{C}^{2n+2-k} \res \Sigma= \mathcal{S}^{2n+2-k} \res \Sigma.
\eeq
\end{teo}
\begin{proof}
By Theorem~\ref{t:pallaconvex} and Theorem~\ref{Maintheoremcentrato}, for every $x \in \Sigma$ we have
$$\beta_d(\ker(Df(x))= \mathcal{H}^{2n+1-k}(\ker(Df(x)) \cap \mathbb{B}(0,1))= \Theta^{*2n+2-k}( \mu, x).$$
Then the area formulas \eqref{eq:areadgenerica} and \eqref{eq:areaC} conclude the proof.
\end{proof}

\bibliography{biblio_Area}
\bibliographystyle{plain}

\end{document}

%% file: def-Area-Heisenberg.tex
\renewcommand{\H}{\mathbb{H}}
\newcommand{\B}{\mathbb{B}}

\newcommand{\M}{\mathbb{M}}

\newcommand{\R}{\mathbb{R}}
\renewcommand{\S}{\mathbb{S}}

\newcommand{\V}{\mathbb{V}}
\newcommand{\W}{\mathbb{W}}

\newcommand{\cB}{\mathcal{B}}

\newcommand{\cH}{\mathcal{H}}

\newcommand{\cL}{\mathcal{L}}
\newcommand{\cN}{\mathcal{N}}

\newcommand{\cS}{\mathcal{S}}

\newcommand{\up}{\upsilon}
\newcommand{\sm}{\setminus}

\newcommand{\diam}{\mbox{\rm diam}}

\newcommand{\spn}{\mbox{\rm span}}

\newcommand{\Lie}{\mathrm{Lie}}

\newcommand{\res}{\mbox{\LARGE{$\llcorner$}}}

\newcommand{\der}{\partial}

\newcommand{\bu}{\mbox{\bf 1}}

\DeclareMathOperator{\Tan}{Tan}

\newcommand{\beqas}{\begin{eqnarray*}}
\newcommand{\eeqas}{\end{eqnarray*}}
\newcommand{\beqa}{\begin{eqnarray}}
\newcommand{\eeqa}{\end{eqnarray}}
\newcommand{\beq}{\begin{equation}}
\newcommand{\eeq}{\end{equation}}
\newcommand{\bce}{\begin{center}}
\newcommand{\ece}{\end{center}}

\newcommand{\pa}[1]{\left( #1 \right)}               
\newcommand{\set}[1]{\left\{ #1 \right\}}            
\newcommand{\ban}[1]{\left\langle  #1 \right\rangle}  

\newtheorem{The}{Theorem}[section]
\newtheorem{Lem}[The]{Lemma}
\newtheorem{Def}[The]{Definition}
\newtheorem{Pro}[The]{Proposition}
\newtheorem{Cor}[The]{Corollary}
\newtheorem{Exa}[The]{Example}
\newtheorem{Con}{Conjecure}

\newcommand{\bt}{\begin{The}}
\newcommand{\et}{\end{The}}
\newcommand{\bl}{\begin{Lem}}
\newcommand{\el}{\end{Lem}}
\newcommand{\bd}{\begin{Def}\rm}
\newcommand{\ed}{\end{Def}}
\newcommand{\br}{\begin{Rem}\rm}
\newcommand{\er}{\end{Rem}}
\newcommand{\bpr}{\begin{Pro}}
\newcommand{\epr}{\end{Pro}}
\newcommand{\bc}{\begin{Cor}}
\newcommand{\ec}{\end{Cor}}
\newcommand{\bj}{\begin{Con}}
\newcommand{\ej}{\end{Con}}
\newcommand{\bex}{\begin{Exa}}
\newcommand{\eex}{\end{Exa}}